\documentclass[12pt,a4paper,reqno]{amsart}

\usepackage{amsfonts}
\usepackage{amsmath}
\usepackage{amssymb}
\usepackage{amsthm}
\usepackage{amsxtra}
\usepackage{bbm}
\usepackage{braket}
\usepackage{comment}
\usepackage{extarrows}
\usepackage{latexsym}
\usepackage{mathrsfs}
\usepackage{mathtools} 
\let\underbrace\LaTeXunderbrace

\usepackage{verbatim}
\usepackage{xcolor}

\usepackage[pdfborder={0 0 0}]{hyperref} 
\hypersetup{breaklinks, raiselinks}

\usepackage{bm}

\usepackage[initials,lite]{amsrefs} 
\AtBeginDocument{%
    \def\MR#1{}
}

\usepackage{cleveref}

\usepackage{a4wide}
\theoremstyle{plain}
\newtheorem{Theorem}{Theorem}[section]
\newtheorem{Lemma}[Theorem]{Lemma}
\newtheorem{Corollary}[Theorem]{Corollary}
\newtheorem{Proposition}[Theorem]{Proposition}

\theoremstyle{definition}

\newtheorem{Assumptions and Discussion}[Theorem]{Assumptions and Discussion}

\newtheorem{Example}[Theorem]{Example}
\newtheorem{Definition}[Theorem]{Definition}

\newtheorem{Question}[Theorem]{Question}

\newtheorem{Remark}[Theorem]{Remark}

\theoremstyle{remark}

\newtheorem{Setting}[Theorem]{Setting}

\newtheorem*{acknowledgment*}{Acknowledgment}

\usepackage[shortlabels]{enumitem}
\SetEnumerateShortLabel{a}{\textup{(\alph*)}}
\SetEnumerateShortLabel{A}{\textup{(\Alph*)}}
\SetEnumerateShortLabel{1}{\textup{(\arabic*)}}
\SetEnumerateShortLabel{i}{\textup{(\roman*)}}
\SetEnumerateShortLabel{I}{\textup{(\Roman*)}}
\setlist{leftmargin=*}

\def\alert#1{{\textcolor{red}{#1}}}

\def\ceil#1{\left\lceil #1 \right\rceil}

\def\deg{\operatorname{deg}}
\def\depth{\operatorname{depth}}

\def\dim{\operatorname{dim}}

\def\floor#1{\left\lfloor #1 \right\rfloor}

\def\Ht{\operatorname{ht}} 

\def\KK{{\mathbb K}}

\def\Min{\operatorname{Min}}

\def\PP{{\mathbb P}}

\def\sat{\operatorname{sat}}
\def\sdefect{\operatorname{sdefect}} 

\def\SS{\mathbb{S}}

\def\supp{\operatorname{supp}}

\def\type{\operatorname{type}}

\newcommand\bdbeta{{\bm \beta}}

\newcommand\bdF{{\bm F}}

\newcommand\bdlambda{{\bm \lambda}}

\newcommand\bdz{{\bm z}}

\newcommand\calA{\mathcal{A}}

\newcommand\calF{\mathcal{F}}

\newcommand\calH{\mathcal{H}}

\newcommand\calL{\mathcal{L}}

\newcommand\calP{\mathcal{P}}

\newcommand\frakm{\mathfrak{m}}

\newcommand\frakp{\mathfrak{p}}

\newcommand\frakq{\mathfrak{q}}

\def\frakS{\mathfrak{S}}

\newcommand\Supp{\operatorname{Supp}}

\newcommand{\Ass}{\operatorname{Ass}}

\newcommand{\projdim}{\operatorname{proj\,dim}}

\def\reg{\operatorname{reg}}

\begin{document}

\title[Generalized Star Configurations]{Symbolic powers and free resolutions of generalized star configurations of hypersurfaces}

\author[Kuei-Nuan Lin, Yi-Huang Shen]{Kuei-Nuan Lin and Yi-Huang Shen}


\thanks{2010 {\em Mathematics Subject Classification}.
    13A15, 
    13A50, 
    13D02, 
    14N20, 
    52C35. 
}

\thanks{Keyword: 
    Betti numbers, fold products, linear quotients, star configuration, symbolic power
}

\address{The Penn State University, Department of Mathematics, Greater Allegheny Campus, McKeesport, PA, 15132, USA}
\email{kul20@psu.edu}

\address{Key Laboratory of Wu Wen-Tsun Mathematics, Chinese Academy of Sciences, School of Mathematical Sciences, University of Science and Technology of China, Hefei, Anhui, 230026, P.R.~China}
\email{yhshen@ustc.edu.cn} 

\begin{abstract}
    As a generalization of the ideals of star configurations of hypersurfaces, we consider the $a$-fold product ideal $I_a(f_1^{m_1}\cdots f_s^{m_s})$ when ${f_1,\dots,f_s}$ is a sequence of generic forms and $1\le a\le m_1+\cdots+m_s$.  Firstly, we show that this ideal has complete intersection quotients when these forms are of the same degree and essentially linear.  Then we study its symbolic powers while focusing on the uniform case with $m_1=\cdots=m_s$.  For large $a$, we describe its resurgence and symbolic defect. And for general $a$, we also investigate the corresponding invariants for meeting-at-the-minimal-components version of symbolic powers.
\end{abstract}

\maketitle

\section{Introduction}
Let $\KK$ be an infinite field.  Partly due to the rich combinatorial structure and the ability of exhibiting extremal numerical behavior, star configurations of points in $\PP_\KK^n$ have attracted strong research interest. As a generalization, the next step is to consider certain union of complete intersection subschemes obtained by intersecting some hypersurfaces in $\PP_\KK^n$.

To be accurate, let $R=\KK[x_0,\dots,x_n]$ be a standard graded polynomial ring over $\KK$ and denote its graded maximal ideal by $\frakm$.  Let $\calF=\{f_1,\dots,f_s\}$ be a set of forms in $R$ with $s\ge n+1$ and consider the hypersurfaces $\calH=\{H_1,\dots,H_s\}$ defined by them in $\PP_{\KK}^n$. Suppose that these forms are \emph{$c$-generic} in the sense that any subset of size at most $c+1$ will form a regular sequence. Then, we will obtain a \emph{star configuration of hypersurfaces} of codimension $c$:
\[
    V_c(\calH,\PP_{\KK}^n)\coloneqq \bigcup_{1\le i_1<\cdots<i_c\le s}
    \left(H_{i_1}\cap \cdots \cap H_{i_c}\right).
\]

It has been of great interest to study various algebraic, geometric and combinatorial properties of star configurations; see for instance \cite{arXiv:1907.04288}, \cite{MR3003727}, \cite{MR3683102}, \cite{arXiv:1907.08172} and the references therein.  The object we are mostly interested here is the defining ideal 
\begin{equation}
    I_{c,\calF} \coloneqq \bigcap_{1\le i_1<\cdots<i_c\le s}\braket{f_{i_1},\dots,f_{i_c}}. 
    \label{def:star-configuration-ideal} 
\end{equation}
It was observed by Geramita et al.~in \cite{MR3683102} that the study of the minimal graded free resolutions of $I_{c,\calF}$ and its symbolic powers can be reduced to the linear monomial case, i.e., when $s=n+1$ and $\calF=\{x_0,x_1,\dots,x_n\}$.  Consequently, these symbolic power ideals will have linear-like resolutions.  

This idea was later made more precise by Mantero in \cite{arXiv:1907.08172} by the notion of Koszul stranded Betti table; see \Cref{def:Koszul-table} below.  Actually he showed that these ideals have complete intersection quotients.  Independently, in the linear monomial case, Biermann et al.~\cite{arXiv:1907.04288} showed that these symbolic power ideals are symmetry strongly shifted and consequently have linear quotients. 

We want to remark at this moment that the ideal $I_{c,\calF}$ above is the specialization of the Stanley--Reisner ideal of a matroid complex. In particular, it is Cohen--Macaulay and consequently a basic double $G$-linkage technique can be applied. Notice that this is crucial in the discussions of \cite{MR3003727} and \cite{MR3683102}, since it allows one to describe minimal generating sets of the ideal $I_{c,\calF}$ and its symbolic powers.  And this is of course the starting point of the description of minimal free resolutions.

The defining ideal $I_{c,\calF}$ in \eqref{def:star-configuration-ideal} can also be studied from a different point of view. Fix positive integers $a,m_1,\dots,m_s$ such that $a\le m_1+\cdots+m_s$. Now, consider the ideal generated by the $a$-fold products of the forms $f_1,\dots,f_s$ with the multiplicities $m_1,\dots,m_s$ respectively:
\begin{equation*}
    I_a(f_1^{m_1}\cdots f_s^{m_s})\coloneqq \Braket{f_1^{n_1}\cdots f_s^{n_s} : 0\le n_i\le m_i \text{ for each $i$ such that }\sum_{i}n_i=a}.
    \label{eqndef:a-fold-product-ideal}
\end{equation*}
The ideal defined above is known as the \emph{$a$-fold product ideal} of the corresponding configuration
\[
    \underbrace{f_1,\dots,f_1}_{m_1},\underbrace{f_2,\dots,f_2}_{m_2},\dots,\underbrace{f_s,\dots,f_s}_{m_s}.
\]
It follows from \cite[Theorem 3.3 and Example 3.4]{MR3683102} that $I_{c,\calF}$ coincides with $I_{s-c+1}(f_1f_2\cdots f_s)$ for $m_1=\cdots=m_s=1$.

The $a$-fold product ideal of linear forms was originally introduced as a nice tool for determining the minimum distance of linear codes in the coding theory.  They also emerge naturally when dealing with higher order Orlik--Terao algebra of hyperplane arrangements.  Intricate algebraic and combinatorial properties of $I_{a}(f_1^{m_1}\cdots f_s^{m_s})$ have also attracted the attention of many researchers; see for instance \cite{zbMATH06759435},  \cite{MR3864202}, \cite{MR2480571}, \cite{MR2679386}, \cite{MR3134006} and the references therein.  

We will pay a closer attention to the ideal $I_a(f_1^{m_1}\cdots f_s^{m_s})$ within this note.  The first aim is to study its Betti table.  For this purpose, throughout this note, we will always assume the following assumption.

\begin{Setting}
    \label{main-setting}
    Let $\calF=\{f_1,\dots,f_s\}$ be a set of forms in $R=\KK[x_0,\dots,x_n]$ and suppose that these forms are \emph{generic}, i.e., they are $n$-generic. Furthermore, fix some positive integers $a,m_1,\dots,m_s$ with $a\le m_1+\cdots+m_s$.
\end{Setting}

Recall that in \cite[Definition 6.5]{arXiv:1907.08172} Mantero introduced the following notion.

\begin{Definition}
    \label{def:Koszul-table}
    Let $1\le n_1<\cdots<n_r$ be integers and $I$ a homogeneous ideal in $R$ generated in degrees $n_1,\dots,n_r$. We say that $I$ has a \emph{Koszul stranded Betti table} if and only if there exists a positive integer $d$ such that the graded Betti number
    \[
        \beta_{i,j}(R/I)\ne 0 \quad \text{only if}\quad j\in \Set{n_h+d(i-1):h=1,\dots,r}
    \]
    for $i\ge 1$.
\end{Definition}

Notice that if $r=d=1$ in the above definition, then the ideal $I$ has a linear resolution. Meanwhile, under the \Cref{main-setting}, if all the forms in $\calF$ are linear, then it is already known by Toh\v{a}neanu and Xie \cite[Theorem 2.3]{arXiv:1906.08346} that the ideal $I_a(f_1^{m_1}\cdots f_s^{m_s})$ has a linear resolution. Inspired by this achievement, we are interested in the following question.

\begin{Question}
    \label{ques-Koszul-Betti}
    Under the \Cref{main-setting}, is it true that the ideal $I_a(f_1^{m_1}\cdots f_s^{m_s})$ has a Koszul stranded Betti table?
\end{Question}

Meanwhile, inspired by the recent work of Toh\v{a}neanu and his coauthors, we are also interested in the following questions.

\begin{Question}
    \label{colon-conj}
    Under the \Cref{main-setting}, is it true that
    \begin{equation}
        I_{a+1}(f_1^{m_1}\cdots f_{s-1}^{m_{s-1}}f_s^{m_s+1}):f_s= 
        I_{a}(f_1^{m_1}\cdots f_{s-1}^{m_{s-1}}f_s^{m_s})
        \label{eqn:colon-ideal-conj}
    \end{equation}
    for all $a\le m_1+\cdots+m_s$?
\end{Question}

\begin{Question}
    \label{main-conj}
    Under the \Cref{main-setting}, is it true that
    \[
        I_a(f_1^{m_1}\cdots f_s^{m_s})= \bigcap_{c=1}^{s} \left( 
            \bigcap_{1\le i_1<\cdots < i_c\le s} \braket{f_{i_1},\dots,f_{i_c}}^{\mu_a(i_1,\dots,i_c)}
        \right),
    \]
    for 
    \begin{equation}
        \mu_a(i_1,\dots,i_c)\coloneqq\max\Set{0, a-\sum_{j\in [s]\setminus \{i_1,\dots,i_c\}}m_j}?
        \label{def:mu_a}
    \end{equation}
    Here, any ideal with non-positive power is replaced by the ring $R$. And as usual, $[s]$ is the set $\{1,2,\dots,s\}$. 
\end{Question}

Regarding the \Cref{colon-conj}, when all the forms of $\calF$ are linear, Anzis, Garrousian and Toh\v{a}neanu showed in \cite[Proposition 3.5]{zbMATH06759435} that the $a$-fold product ideal having a linear resolution is closely related to the colon ideal having the expected form as in \eqref{eqn:colon-ideal-conj}, even without the generic assumption.  If the linear forms in $\calF$ are generic, then Toh\v{a}neanu and Xie provided a positive answer to the \Cref{colon-conj} in the proof of \cite[Theorem 2.3]{arXiv:1906.08346}, as a byproduct of establishing the linear resolution property of the $a$-fold product ideal.  For their proof, a positive answer to the \Cref{main-conj}  is vital; cf.~\cite[Corollary 2.4]{arXiv:1906.08346}. And they used this to confirm a conjecture of Geramita, Harbourne and Migliore (\cite[Conjecture 4.1]{MR3003727}) regarding a primary-decomposition type formula for powers of ideals of star configuration of hyperplanes.

Indeed, Geramita, Harbourne and Migliore themselves proved their own conjecture up to saturation in \cite[Corollary 4.9]{MR3003727}. Using a similar technique, we will show in \Cref{prop:sat} that a similar phenomenon also happens for generic configurations of hypersurfaces. It is worth pointing out at this stage that the ideal $I_a(f_1^{m_1}\cdots f_s^{m_s})$ is not unmixed in general (even when $m_1=\cdots=m_s$), let alone Cohen--Macaulay.

Notice that with regard to general ideals lacking mutigrading structure, even for monomial ideals in the generic forms in $\calF$, weird phenomena emerge for the descriptions of associated primes, containment problems and the colon operations; see for example \cite[Remark 3.2]{arXiv:1907.08172}. Therefore, one won't be surprised to see computations by the software \texttt{Macaulay2} \cite{M2} showing that the three questions raised above have negative answers in general. Therefore, we are obliged to impose a technical condition as being \emph{strongly generic}. This new term means that in addition to $\calF$ being generic and the forms there sharing a common degree, the forms actually come from a common $\KK$-linear space of $\KK$-dimension $\dim(R)=n+1$. Although this assumption seems artificial, it is still natural in the sense that generic set of linear forms are strongly generic. Notice that the latter is the condition needed when we are dealing with star configurations of hyperplanes.  

To some extent, the strongly generic terminology simply means being generic and essentially linear. And when this strong condition is satisfied, we show with ease in \Cref{thm:sg-1} that all the three questions above have positive answers. 

As the first main contribution of this note, we indeed prove in \Cref{thm:strongly-generic-ciq} that the $a$-fold product ideal $I=I_a(f_1^{m_1}\cdots f_s^{m_s})$ of generic forms will have complete intersection quotients under the strongly generic condition. This will particularly imply that the $a$-fold product ideal of generic \emph{linear} forms has a linear resolution, recovering a key result of Toh\v{a}neanu and Xie in \cite{arXiv:1906.08346}. Our approach is inspired by the viewpoint of Geramita, Harbourne, Migliore and Nagel in \cite{MR3683102} and the recent work of Mantero in \cite{arXiv:1907.08172}.  Since we don't have to dwell on the primary-decomposition type problems, this approach is more straightforward, and hence shorter when compared with the work in \cite{arXiv:1906.08346}. 

The argument for the complete intersection quotients actually bears more fruit than one generally expects at the first glance. It allows us to compare its Betti table with that of the prototype monomial ideal, i.e., when the forms $f_1,\dots,f_s$ are actually ring variables. In particular, an upper bound of the projective dimension of the ideal $I_a(f_1^{m_1}\cdots f_s^{m_s})$ can be easily described; see our \Cref{cor:compare-projdim}. To obtain such a comparison, we don't need the strongly generic assumption in some nice cases. And this benefits us a lot.

In the rest of the note, we focus on the uniform subclass $I_a(f_1^b\cdots f_s^b)$ under some sufficient generic assumptions. 
In other words, we don't require it to be $n$-generic, let alone the strong genericness.
Notice that in the star configuration case, $b=1$.  Therefore, we can also call it the defining ideal of \emph{generalized star configuration of hypersurfaces}.  Whence, the aforementioned information of the projective dimension will play an important role. Recall that we already have a primary-decomposition type formula for the saturation of $I_a(f_1^b\cdots f_s^b)$. After sealing these two parts together, we can handle the associated primes of $I_a(f_1^b \cdots f_s^b)$ more accurately. In particular, we will feel more at ease when manipulating the symbolic powers of $I_a(f_1^b\cdots f_s^b)$. 

Recall that if $I$ is a homogeneous ideal of a Noetherian domain $R$, for every positive integer $m$, the \emph{$m$-th symbolic power} of $I$ is the homogeneous ideal
\begin{equation}
    I^{(m)}\coloneqq \bigcap_{\frakp\in\Ass(R/I)}I^m R_{\frakp}\cap R.
    \label{eqn:symbolic-power}
\end{equation} 
For instance, when $I$ is the defining ideal of a reduced affine scheme over an algebraically closed field of characteristic zero, Zariski and Nagata showed that $I^{(m)}$ is generated by the polynomials whose partial derivatives of orders up to $m-1$ vanish on this scheme; c.f.~\cite{MR535850}.  
The research of related topics has continuously attracted the eyes of many researchers; see for instance the recent survey \cite{MR3779569} and the references therein. 

Regarding symbolic powers, the two constants that we are mostly interested in here are the resurgence and the symbolic defect. Recall that the \emph{resurgence} of $I$ is defined to be 
\begin{equation*}
    \rho(I)\coloneqq \sup\Set{\frac{m}{r}: I^{(m)}\nsubseteq I^r}.
\end{equation*}
It was pointed out in \cite{MR3683102} that there are very few results determining the resurgence of the ideal of a subscheme whose dimension is at least one and whose codimension is at least two, apart from ideal of cones and certain monomial ideals. 

Meanwhile, the \emph{symbolic defect} of $I$, as the first estimate of the size of the symbolic power, is defined to be
\begin{equation*}
    \sdefect(I,m)\coloneqq \mu(I^{(m)}/I^m)
\end{equation*}
for each positive integer $m$. Here, $\mu$ gives the minimal number of generators of the corresponding graded module. Symbolic defect of star configurations has recently been studied by \cite{arXiv:1907.04288}, \cite{MR3906569} and \cite{arXiv:1907.08172}, to name a few. 

Based on the foundation laid above, we are able to scrutinize the symbolic powers of $I_a(f_1^b\cdots f_s^b)$ when $a$ is large. As an application, we give closed formulas for the resurgence and symbolic defects in \Cref{Cor:sdefect-rho-strongly-generic}. If paying more attention to these two quantities, then one realizes that a delightful description of a minimal generating set of the symbolic powers of $I_a(f_1^b\cdots f_s^b)$ is inevitable. This is accomplished in \Cref{thm:mingens-strongly-generic}. It is worth pointing out here that the corresponding part in \cite{MR3683102} depends heavily on the application of the basic double $G$-linkage technique, which requires the Cohen--Macaulay condition. But the latter is not desirable here in general. And this embodies the second main contribution of this note.

For more general $a$, the uniform $a$-product ideal $I_a(f_1^b\cdots f_s^b)$ will not be saturated. Whence, the symbolic powers coincide with corresponding standard powers, making the direct probe of this topic dull.  Therefore, we take a different path. Notice that there is another algebraic approach in the literature for treating symbolic objects, namely, one only takes intersections over minimal primes of $I$ on the right side of \eqref{eqn:symbolic-power}; see for instance \cite[Definition 4.3.22]{MR3362802}. To avoid confusion in notation, we will call it the \emph{$m$-th $*$-symbolic power} of $I$, namely, we will define
\begin{equation*}
    I^{(m)_*}\coloneqq \bigcap_{\frakp\in\Min(I)} (I^m R_\frakp \cap R).
\end{equation*}
Numous studies were also devoted to the symbolic powers along this line; see, for instance, the recent papers \cite{arXiv:1702.01766} and \cite{MR4022079}.  

Regarding symbolic powers of this flavor, the terminologies of \emph{$*$-resurgence} and \emph{$*$-symbolic defect} emerge naturally as well. In contrast, the argument involved for determining these two constants is similar, but more direct. And the outcome is included in the last section of this paper. As an unexpected harvest, the $*$-symbolic defect encodes information regarding whether $I_a(f_1^b\cdots f_s^b)$ is a power of the star configuration ideal $I_{c,\calF}$ for some $c$; see \Cref{prop:compare-power-star-symbolic-power}.  

We also want to point out in the end of this section that Conca and Tsakiris \cite{arXiv:1910.01955} recently considered the ideal of subspace arrangements, which is an ideal of fold product of different flavor. Under some generic condition, the ideal they considered can also be linked to discrete polymatroids. In particular, the ideal will have linear quotients.

\section{Generic case}
The main result of this section is the primary-decomposition type formula of the $a$-fold product ideal $I_a(f_1^{m_1}\cdots f_s^{m_s})$ up to saturation in \Cref{prop:sat}. This result has two happy consequences. Firstly, if the $a$-fold product ideal is a priori saturated, then it leads to the expected formula that we are seeking. Secondly, if the configuration is uniform in the sense that $m_1=\cdots=m_s$, then it paves the path for our exploration of the symbolic powers of this $a$-fold product ideal. Both are needed in later sections.

\begin{Remark}
    \label{key-facts}
    Throughout this note, we will make substantial use of the following facts.
    \begin{enumerate}[a]
        \item Let $T=\KK[z_1,\dots,z_s]$ and $R=\KK[x_0,\dots,x_n]$ be two polynomial rings over a field $\KK$. Let $f_1,\dots,f_s$ be an $R$-regular sequence of homogeneous elements.  Define a homomorphism $\varphi:T\to R$ induced by $\varphi(z_i)=f_i$ for $1\le i\le s$. Then $\varphi$ is flat.  Furthermore, if $I$ is a monomial ideal in $T$ and $\bdF_{\bullet}$ is a graded minimal free resolution of $T/I$ over $T$. Then $\bdF_{\bullet}\otimes R$ is a graded minimal free resolution of $R/\varphi(I)$ over $R$. In particular, $I$ and the ideal $\varphi(I)$ have the same graded Betti numbers over $T$ and $R$ respectively, except possibly with shifts which depend on the degrees of the $f_i$.  This fact was stated in \cite[Lemma 3.1]{MR3683102} and its proof.
        \item If a homomorphism $\varphi:T\to R$ between Noetherian rings is
            flat, then for arbitrary ideals $I_1,I_2$ in $T$, we will have
            $\varphi(I_1\cap I_2)=\varphi(I_1)\cap\varphi(I_2)$ as well as
            $\varphi(I_1:I_2)=\varphi(I_1):\varphi(I_2)$, by \cite[Theorem
            7.4]{MR879273}. 
    \end{enumerate}
\end{Remark}

To begin our voyage of generalized star configuration of hypersurfaces, we first notice that both \Cref{colon-conj} and \Cref{main-conj} have positive answers for monomial configurations, i.e., when $\calF$ is a set of ring variables. Indeed, we have

\begin{Lemma}
    \label{lem:reg-case}
    If $f_1,\dots,f_s$ is an $R$-regular sequence of homogeneous elements in
    $R=\KK[x_0,\dots,x_n]$, then
    \[
        I_a(f_1^{m_1}\cdots f_s^{m_s})= \bigcap_{c=1}^{s} \left( 
            \bigcap_{1\le i_1<\cdots < i_c\le s} \braket{f_{i_1},\dots,f_{i_c}}^{\mu_a(i_1,\dots,i_c)}
        \right)
    \]
    and
    \[
        I_{a+1}(f_1^{m_1}\cdots f_{s-1}^{m_{s-1}}f_s^{m_s+1}):f_s= 
        I_{a}(f_1^{m_1}\cdots f_{s-1}^{m_{s-1}}f_s^{m_s}).
    \]
    Furthermore, the ideal $I_a(f_1^{m_1}\cdots f_s^{m_s})$ has a Koszul stranded Betti table. 
\end{Lemma}

Recall that the notation $\mu_a(i_1,\dots,i_c)$ was explained previously in
\eqref{def:mu_a}. And the assumption of $s\ge n+1$ is not required here.  

\begin{proof}
    The statements hold when $f_1,\dots,f_s$ is the sequence of variables
    $x_0,\dots,x_{s-1}$; see for instance \cite[Theorem 2.3 and Corollary
    2.4]{arXiv:1906.08346} and their proofs. 

    For the general case, consider the following treatment as in
    \cite{MR3683102}. Let $T=\KK[z_1,\dots,z_s]$ and consider the ring
    homomorphism 
    \[
        \varphi: T\to R,\quad z_i\mapsto f_i.
    \]
    As $\varphi$ is flat, it remains to apply the facts mentioned in \Cref{key-facts}. 
\end{proof}

\begin{Lemma}
    \label{I-containment}
    Under the \Cref{main-setting}, we have
    \[
        I_a(f_1^{m_1}\cdots f_s^{m_s})\subseteq \bigcap_{c=1}^{n} \left( 
            \bigcap_{1\le i_1<\cdots < i_c\le s} \braket{f_{i_1},\dots,f_{i_c}}^{\mu_a(i_1,\dots,i_c)}
        \right).
    \]
\end{Lemma}

\begin{proof}
    It is not difficult to argue as in the proof of \cite[Lemma 2.1]{arXiv:1906.08346}.
\end{proof} 

\begin{Theorem}
    \label{prop:sat}
    Under the \Cref{main-setting}, the equality in \Cref{main-conj} holds up to
    saturation, i.e.,
    \begin{equation}
        I_a(f_1^{m_1}\cdots f_s^{m_s})^{\sat}=
        \bigcap_{c=1}^{n} \left( 
            \bigcap_{1\le i_1<\cdots < i_c\le s} \braket{f_{i_1},\dots,f_{i_c}}^{\mu_a(i_1,\dots,i_c)}
        \right).
        \label{eqn-sat}
    \end{equation}
\end{Theorem}

\begin{proof}
    We will follow the technique in the proofs of \cite[Corollary 4.9]{MR3003727} and \cite[Proposition 2.2]{arXiv:1906.08346}. Regarding the equality in \eqref{eqn-sat}, we have at least the containment $\text{LHS}\subseteq\text{RHS}$, by \Cref{I-containment}. Thus, it suffices to take any graded non-maximal ideal $\frakp$ containing $I_a(f_1^{m_1}\cdots f_s^{m_s})$ and show that the equality in \eqref{eqn-sat} holds locally with respect to $\frakp$.

    First of all, one can verify with ease that
    \[
        I_a(f_1^{m_1}\cdots f_s^{m_s})=\sum_{k=0}^{m_s}f_s^{k}I_{a-k}(f_1^{m_1}\cdots f_{s-1}^{m_{s-1}}).
    \]
    Since
    \[
        I_{a-m_s}(f_1^{m_1}\cdots f_{s-1}^{m_{s-1}})\supseteq \cdots \supseteq  I_{a-k}(f_1^{m_1}\cdots f_{s-1}^{m_{s-1}})\supseteq \cdots \supseteq I_{a}(f_{1}^{m_1}\cdots f_{s-1}^{m_{s-1}})
    \]
    for $0\le k\le m_s$, we will have
    \begin{equation}
        I_a(f_1^{m_1}\cdots f_s^{m_s})R_{\frakp}=I_{a-m_s}(f_1^{m_1}\cdots f_{s-1}^{m_{s-1}})R_{\frakp},
        \label{eqn:local-product}
    \end{equation}
    once $f_{s}\notin \frakp$.

    Now, without loss of generality, we may assume that $f_1,\dots,f_r\in \frakp$ while $f_{r+1},\dots,f_s\notin \frakp$. Because of the generic assumption on $\calF$, we have that $r\le \Ht(\frakp)<\Ht(\frakm)=n+1$. It is clear that
    \[
        \bigcap_{c=1}^{\min(s,n)} \left( 
            \bigcap_{1\le i_1<\cdots < i_c\le s} \braket{f_{i_1},\dots,f_{i_c}}^{\mu_a(i_1,\dots,i_c)}
        \right)R_\frakp=
        \bigcap_{c=1}^{r} \left( 
            \bigcap_{1\le i_1<\cdots < i_c\le r} \braket{f_{i_1},\dots,f_{i_c}}^{\mu_a(i_1,\dots,i_c)}
        \right)R_\frakp,
    \]
    while
    \[
        I_a(f_1^{m_1}\cdots f_s^{m_s})^{\sat} R_\frakp 
        = I_a(f_1^{m_1}\cdots f_s^{m_s}) R_\frakp
        = I_{a-\sum_{c=r+1}^{s}m_c}(f_1^{m_1}\cdots f_r^{m_r}) R_\frakp
    \]
    by applying the equality in \eqref{eqn:local-product} repeatedly.

    As the last step, it remains to verify that
    \[
        I_{a-\sum_{c=r+1}^{s}m_c}(f_1^{m_1}\cdots f_r^{m_r})
        =
        \bigcap_{c=1}^{r} \left( 
            \bigcap_{1\le i_1<\cdots < i_c\le r} \braket{f_{i_1},\dots,f_{i_c}}^{\mu_a(i_1,\dots,i_c)}
        \right).
    \]
    But this is exactly what we have shown in \Cref{lem:reg-case}.
\end{proof}

The decomposition in \eqref{eqn-sat} can be refined once we know more about the associated primes of $I_a(f_1^{m_1}\cdots f_s^{m_s})$.

\begin{Theorem}
    \label{prop:dec-c-generic} 
    Let $\calF=\{f_1,\dots,f_s\}$ be a set of $\hat{c}$-generic forms in $R=\KK[x_0,\dots,x_n]$. Let $a,m_1,\dots,m_s$ be positive integers with $a\le m_1+\cdots+m_s$. Suppose that $\Ht(\frakq)\le \hat{c}\le n$ for each $\frakq\in \Ass(R/I_a(f_1^{m_1}\cdots f_s^{m_s}))$. Then we have
    \begin{equation*}
        I_a(f_1^{m_1}\cdots f_s^{m_s})=
        \bigcap_{c=1}^{\hat{c}} \left( 
            \bigcap_{1\le i_1<\cdots < i_c\le s} \braket{f_{i_1},\dots,f_{i_c}}^{\mu_a(i_1,\dots,i_c)}
        \right).
    \end{equation*}
\end{Theorem}

\begin{proof}
    The proof is similar to that of \Cref{prop:sat}. And one only needs to take prime ideals $\frakp$ of $R$ such that $\Ht(\frakp)\le \hat{c}$, by our assumption.
\end{proof}

\begin{Remark}
    \label{rmk:ass-cond}
    It follows from \cite[Proposition 1.2.13 and Theorem 1.3.3]{MR1251956} that $\Ht(\frakq)\le \projdim(R/I_a(f_1^{m_1}\cdots f_s^{m_s}))$ for each $\frakq\in \Ass(R/I_a(f_1^{m_1}\cdots f_s^{m_s}))$. Therefore, to apply \Cref{prop:dec-c-generic} efficiently, one needs to bound the projective dimension of $I_a(f_1^{m_1}\cdots f_s^{m_s})$.
\end{Remark}

Now, consider the defining ideal of the star configuration of hypersurfaces in
$\PP^n$
\[
    I_{c,\calF}\coloneqq \bigcap_{1\le i_1<\cdots < i_c\le s} \braket{f_{i_1},\dots,f_{i_c}}.
\]
It was shown in \cite[Theorem 3.3]{MR3683102} that 
\begin{equation}
    I_{c,\calF}=I_{s-c+1}(f_1\cdots f_s).
    \label{eqn:I-I}
\end{equation}
Indeed, it was shown in \cite[Theorem 3.6]{MR3683102} that for each positive
integer $m$, the symbolic power
\begin{equation}
    I_{c,\calF}^{(m)}=\bigcap_{1\le i_1<\cdots < i_c\le s} \braket{f_{i_1},\dots,f_{i_c}}^m.
    \label{eqn:symbolic-power-star-configuration}
\end{equation}
Meanwhile, it is clear that
\begin{equation}
    I_a(f_1^{m_1}\cdots f_s^{m_s})^m=I_{ma}(f_1^{mm_1}\cdots f_s^{mm_s})
    \label{eqn-power}
\end{equation}
for each positive integer $m$. Consequently,
\[
    I_{c,\calF}^m=
    (I_{s-c+1}(f_1\cdots f_s))^m=
    I_{m(s-c+1)}(f_1^m\cdots f_s^m).
\]
Therefore, we obtain the following corollary from \Cref{prop:sat}.

\begin{Corollary}
    Under the \Cref{main-setting}, the saturation of the power of the
    star configuration ideal can be expressed as
    \[
        (I_{c,\calF}^m)^{\sat}=I_{c,\calF}^{(m)} \cap I_{c+1,\calF}^{(2m)} \cap \cdots \cap I_{n,\calF}^{((n-c+1)m)},
    \]
    or, equivalently,
    \[
        (I_{s-c+1}(f_1\cdots f_s)^m)^{\sat} =\bigcap_{c\le k\le n}
        I_{s-k+1}(f_1\cdots f_s)^{( (k-c+1)m)} .
    \]
\end{Corollary}

Indeed, the latest equality can be easily generalized by considering a similar
decomposition of $I_a(f_1^b\cdots f_s^b)^m=I_{ma}(f_1^{mb}f_2^{mb}\cdots
f_s^{mb})$ of positive dimension, for positive integers $a,b$ and $m$.  To facilitate discussions in later sections, we will
fix the following notations.

\begin{Setting}
    \label{setting-uniform-fold-product}
    Let $a,b,c_0,n,s,\mu_a^0$ be positive integers such that 
    \[
        c_0\coloneqq s-\floor{\frac{a-1}{b}}, \quad
        \mu_a^0\coloneqq a-b(s-c_0), \quad s\ge n+1\quad
        \text{and} \quad
        b(s-n)+1\le a\le sb.
    \]
\end{Setting}

Here, we explain a little bit the origin of the notations chosen above.
Under such setting, for each component $\braket{f_{i_1},\dots,f_{i_c}}$ in
\eqref{eqn-sat} with $m_1=\cdots=m_s=b$, the exponent
$\mu_a({i_1},\dots,{i_c})=a-b(s-c)\ge 1$ precisely when $c\ge
s-\floor{\frac{a-1}{b}}=c_0$. Meanwhile,
\begin{equation}
    \mu_a^0=a-b(s-c_0)=a-b\floor{\frac{a-1}{b}}\in\{1,2,\dots,b\}.
    \label{eqn:mua0}
\end{equation}
Whence, the exponent in \eqref{eqn-sat} satisfies
\[
	\mu_a(i_1,\dots,i_c)=a-b(s-c)=a-b(s-c_0)+b(c-c_0)=\mu_a^{0}+b(c-c_0)
\]
for $c_0\le c\le s$. In order that $I_a(f_1^b\cdots f_s^b)$ is not zero dimensional, we have to restrict ourselves to those $a$'s such that $c_0\le n$, i.e., $a\ge b(s-n)+1$. To wrap up this short discussion, we remark that
\begin{equation}
a=b(s-c_0)+\mu_a^0 \qquad \text{with}\quad 1\le \mu_a^0\le b.
\label{eqn:a}
\end{equation}
This non-standard long division might be more intuitive for some readers when verifying various estimates and equalities in later sections.

Now, it follows directly from \Cref{prop:sat} and the equality
\eqref{eqn:symbolic-power-star-configuration} that we have the following
formula for the \emph{uniform $a$-fold product ideal} $I_a(\calF^b)\coloneqq
I_a(f_1^b\cdots f_s^b)$ in $R=\KK[x_0,\dots,x_n]$. We also call it the defining
ideal of \emph{generalized star configuration of hypersurfaces}.

\begin{Corollary}
    With the assumptions in Settings \ref{main-setting} and
    \ref{setting-uniform-fold-product}, we have
    \[
        (I_a(\calF^b)^m)^{\sat}=
        \bigcap_{c_0\le c\le n} I_{c,\calF}^{(m(\mu_a^0+b(c-c_0)))} 
    \]
    for each positive integer $m$.
\end{Corollary}


\section{Strongly generic case}

Unfortunately, but not surprisingly, a computation by the software \texttt{Macaulay2} \cite{M2} shows that 
all of \Cref{ques-Koszul-Betti}, \Cref{colon-conj} and \Cref{main-conj} have negative answers in general for quadratic generic forms in $\KK[x_0,x_1,x_2]$.  

Thus, in the following, due to the above technical obstruction, we have to inflict a stricter condition so that the forms we treat are essentially linear.  Suppose first that the forms in $\calF=\{f_1,\dots,f_s\}\subset R=\KK[x_0,\dots,x_n]$ have a common degree $d$. Then, we say that $\calF$ is \emph{strongly generic} if it is generic and 
\[
    \dim_{\KK}(\KK f_1+\cdots +\KK f_s)= n+1=\dim(R).
\]
Notice that for every $k\ge n+1$, we will have the ideals
\[
    \braket{f_1,\dots,f_{n+1}} = \braket{f_{i_1},\dots,f_{i_k}} \qquad \text{ for any $1\le i_1<\cdots<i_k\le s$}.
\]
This common $\frakm$-primary complete intersection ideal will be denoted by $\frakm_{\calF}$.

\begin{Remark}
    When $d=1$, then the set $\calF$ being strongly generic is equivalent to it being generic.
\end{Remark}

In this section, we will shift our focus to $a$-fold product ideals of strongly generic sets of forms. As a warm-up, we first notice that \Cref{lem:reg-case} can be strengthened and all of \Cref{ques-Koszul-Betti}, \Cref{colon-conj} and \Cref{main-conj} have positive answers in this situation.

\begin{Proposition}
    \label{thm:sg-1}
    Under the \Cref{main-setting}, suppose in addition that
    $\calF=\{f_1,\dots,f_s\}$ is strongly generic. Then, 
    \[
        I_a(f_1^{m_1}\cdots f_s^{m_s})= \bigcap_{c=1}^{s} \left( 
            \bigcap_{1\le i_1<\cdots < i_c\le s} \braket{f_{i_1},\dots,f_{i_c}}^{\mu_a(i_1,\dots,i_c)}
        \right)
    \]
    and
    \[
        I_{a+1}(f_1^{m_1}\cdots f_{s-1}^{m_{s-1}}f_s^{m_s+1}):f_s= 
        I_{a}(f_1^{m_1}\cdots f_{s-1}^{m_{s-1}}f_s^{m_s}).
    \]
    Furthermore, the ideal $I_a(f_1^{m_1}\cdots f_s^{m_s})$ has a Koszul
    stranded Betti table. 
\end{Proposition}

\begin{proof}
    By virtue of \Cref{lem:reg-case}, it suffices to consider the case when $s\ge n+2$. Whence, under the assumptions on $\calF$, for each $j=n+2,\dots,s$, we will have $f_j=\sum_{i=1}^{n+1} \mu_{ji}f_i$ for some $\mu_{ji}\in \KK$. Now consider a new configuration $\calL=\{\ell_1,\dots,\ell_s\}$ in $S=\KK[y_1,\dots,y_{n+1}]$, where $\ell_1=y_1,\dots,\ell_{n+1}=y_{n+1}$, and $\ell_j=\sum_{i=1}^{n+1} \mu_{ji}y_i$ for $j=n+2,\dots,s$. Since $\calF$ is strongly generic, so is $\calL$.  Meanwhile, the desired properties hold if we replace $\calF$ by $\calL$, because of \cite[Theorem 2.3 and Corollary 2.4]{arXiv:1906.08346} and their proofs.

    To finish our proof, consider the homomorphism $\psi:S\to R$ induced by $\psi(y_i)=f_i$ for $1\le i\le n+1$. As $f_1,\dots,f_{n+1}\in R$ form a regular sequence, the map $\psi$ is flat. Notice that $I_a(f_1^{m_1}\cdots f_s^{m_s})=\psi(I_a(\ell_1^{m_1}\cdots \ell_s^{m_s}))$. Now, it remains to apply the facts stated in \Cref{key-facts}.
\end{proof}

Recall that the following formula was first conjectured by Geramita, Harbourne
and Migliore in \cite[Conjecture 4.1]{MR3003727} for any generic set $\calF$ of
linear forms. It was later established by Toh\v{a}neanu and Xie in
\cite[Theorem 3.2]{arXiv:1906.08346}.
\[
    I_{c,\calF}^m=I_{c,\calF}^{(m)} \cap I_{c+1,\calF}^{(2m)} \cap \cdots \cap I_{n,\calF}^{((n-c+1)m)}\cap \frakm^{(s-c+1)m}.
\]
Once we apply the flat argument in the proof of \Cref{thm:sg-1} to the above
formula, we immediately obtain the following result.        

\begin{Proposition}
    Under the \Cref{main-setting}, suppose in addition that
    $\calF=\{f_1,\dots,f_s\}$ is strongly generic.  Then the power of the
    star configuration ideal can be expressed as
    \[
        I_{c,\calF}^m=I_{c,\calF}^{(m)} \cap I_{c+1,\calF}^{(2m)} \cap \cdots \cap I_{n,\calF}^{((n-c+1)m)}\cap\frakm_{\calF}^{(s-c+1)m}.
    \]
\end{Proposition} 

\begin{proof}
    It remains to point out the formula in \eqref{eqn:symbolic-power-star-configuration} regarding the symbolic powers.
\end{proof}

The Koszul stranded Betti table part in \Cref{thm:sg-1} can be strengthened.  Notice that as a generalization of ideals with linear quotients,  Mantero introduced the notion of complete intersection quotients in \cite[Definition 6.1]{arXiv:1907.08172}.

\begin{Definition}
    Let $I$ be a homogeneous ideal in the polynomial ring $R=\KK[x_0,\dots,x_n]$. We say that $I$ has \emph{c.i.~quotients} if there exists a total order $h_1>\cdots>h_r$ on a generating set $\{h_1,\dots,h_r\}$ of $I$, such that for any $1\le i\le r-1$, the colon ideal $\braket{h_1,\dots,h_i}:h_{i+1}$ is a complete intersection ideal. Furthermore, if each such colon ideal has the same degree $d$, then we say that $I$ has \emph{$d$-c.i.~quotients}.
\end{Definition}

It is clear that $I$ has linear quotients precisely when $I$ has
$1$-c.i.~quotients. Furthermore, if $I$ has $d$-c.i.~quotients, then it has a
Koszul stranded Betti table by \cite[Corollary 6.6]{arXiv:1907.08172}.

The following result is crucial for our voyage of generalized star configurations.

\begin{Lemma}
    \label{prop:strongly-generic-ciq}
    Let $I$ be a monomial ideal having linear quotients in $T=\KK[z_1,\dots,z_s]$ and $\calF=\{f_1,\dots,f_s\}$ a strongly generic set of $d$-forms in $R=\KK[x_0,\dots,x_n]$. Regarding the homomorphism $\varphi:T\to R$ induced by $z_i\mapsto f_i$ for each $i$, the specialization $\varphi(I)$ has $d$-c.i.~quotients. 
\end{Lemma}

\begin{proof}
    Since $I$ is monomial and has linear quotients, there is a total order $\bdz^{\bdbeta_1}>\cdots > \bdz^{\bdbeta_r}$ on the minimal monomial generating set $\{\bdz^{\bdbeta_1},\dots,\bdz^{\bdbeta_r}\}$ of $I$ so that each successive colon ideal $\braket{\bdz^{\bdbeta_1},\dots,\bdz^{\bdbeta_i}}:\bdz^{\bdbeta_{i+1}}$ is generated by a subset of the variables $\{z_1,\dots,z_s\}$.

    Now, $\varphi(I)=\braket{\varphi(\bdz^{\beta_1}),\dots,\varphi(\bdz^{\bdbeta_r})}$.  For each $i\le r-1$, if $\varphi(\bdz^{\bdbeta_{i+1}})\notin \varphi(\braket{\bdz^{\bdbeta_1},\dots,\bdz^{\bdbeta_i}})$, we claim that
    \[
        \varphi(\braket{\bdz^{\bdbeta_1},\dots,\bdz^{\bdbeta_i}}):\varphi(\bdz^{\bdbeta_{i+1}})=
        \varphi(\braket{\bdz^{\bdbeta_1},\dots,\bdz^{\bdbeta_i}}:\bdz^{\bdbeta_{i+1}}).
    \]
    Of course, if $\varphi(\bdz^{\bdbeta_{i+1}})\in \varphi(\braket{\bdz^{\bdbeta_1},\dots,\bdz^{\bdbeta_i}})$, then we can simply remove $\varphi(\bdz^{\bdbeta_{i+1}})$ from the generating set of $\varphi(I)$.

    For simplicity, for this fixed $i$, suppose that the quotient ideal $\braket{\bdz^{\bdbeta_1},\dots,\bdz^{\bdbeta_i}}:\bdz^{\bdbeta_{i+1}}$ is generated by $z_{j_1},z_{j_2},\dots,z_{j_t}$ with $1\le j_1<j_2<\cdots<j_t\le s$. Hence, the aforementioned claim asks for
    \begin{equation*}
        \braket{\varphi(\bdz^{\bdbeta_1}),\dots,\varphi(\bdz^{\bdbeta_i})}:\varphi(\bdz^{\bdbeta_{i+1}})=
        \varphi(\braket{z_{j_1},\dots,z_{j_t}}).
    \end{equation*}
    Regarding the latest equality, as the containment $\text{LHS}\supseteq\text{RHS}$
    is clear, it remains to prove the reverse containment
    \begin{equation}
        \braket{\varphi(\bdz^{\bdbeta_1}),\dots,\varphi(\bdz^{\bdbeta_i})}:\varphi(\bdz^{\bdbeta_{i+1}})\subseteq 
        \varphi(\braket{z_{j_1},\dots,z_{j_t}}).
        \label{eqn:colon1}
    \end{equation}
    \begin{enumerate}[a]
        \item When $t\ge n+1$, then the RHS of \eqref{eqn:colon1} is simply $\frakm_{\calF}$. We will use the notations in the proof of \Cref{thm:sg-1}. Whence, there exists forms $g_1,\dots,g_r\in S=\KK[y_1,\dots,y_{n+1}]$ such that $\varphi(\bdz^{\bdbeta_j})=\psi(g_j)$ for each $j$. As $\varphi(\bdz^{\bdbeta_{i+1}})\notin \varphi(\braket{\bdz^{\bdbeta_1},\dots,\bdz^{\bdbeta_i}})$, $g_{i+1}\notin \braket{g_1,\dots,g_i}$. Therefore, $\braket{g_1,\dots,g_i}:g_{i+1}\subseteq \braket{y_1,\dots,y_{n+1}}$, the graded maximal ideal of $S$. Since $\psi$ is flat, this implies that 
            \[
                \psi(\braket{g_1,\dots,g_i}):\psi(g_{i+1})=
                \psi(\braket{g_1,\dots,g_i}:g_{i+1})\subseteq \psi(\braket{y_1,\dots,y_{n+1}})=\frakm_{\calF}.
            \]
            Hence the containment \eqref{eqn:colon1} holds in this case.
        \item When $t\le n$, to prove \eqref{eqn:colon1}, it suffices to prove
            \begin{equation}
                \varphi(\bdz^{\bdbeta_k}):\varphi(\bdz^{\bdbeta_{i+1}})\subseteq
                \braket{f_{j_1},\dots,f_{j_t}}
                \label{eqn:individual-contaiment} 
            \end{equation} 
            for $k=1,\dots,i$, by applying \cite[Proposition 6.11]{arXiv:1907.08172} with $c=n$.  Notice that as $\bdz^{\bdbeta_k}:\bdz^{\bdbeta_{i+1}}\subseteq \braket{z_{j_1},\dots,z_{j_t}}$, 
            \[
                \supp(\bdz^{\bdbeta_k}/\gcd(\bdz^{\bdbeta_k},\bdz^{\bdbeta_{i+1}}))\cap \{j_1,\dots,j_t\}\ne \varnothing.
            \]
            Without loss of generality, say $j_1$ belongs to this intersection. Then,
            \begin{align*}
                \varphi(\bdz^{\bdbeta_k}):\varphi(\bdz^{\bdbeta_{i+1}})&= \varphi(\bdz^{\bdbeta_k}/\gcd(\bdz^{\bdbeta_k},\bdz^{\bdbeta_{i+1}})):\varphi(\bdz^{\bdbeta_{i+1}}/\gcd(\bdz^{\bdbeta_k},\bdz^{\bdbeta_{i+1}}))\\
                &\subseteq  \braket{f_{j_1}} \subseteq \braket{f_{j_1},\dots,f_{j_t}},
            \end{align*}
            where the first containment is due to \cite[Lemma 6.8]{arXiv:1907.08172}. Since this confirms the containment in \eqref{eqn:individual-contaiment}, our proof is finished.
            \qedhere
    \end{enumerate}
\end{proof}

It is clear from the above proof that the additional strongly generic assumption is only needed when dealing with the maximal part $\frakm_{\calF}$. Thus, we will have a variant result under the mild generic condition.

\begin{Lemma}
    \label{prop:c-generic-ciq}
    Let $I$ be a monomial ideal having linear quotients in $T=\KK[z_1,\dots,z_s]$ and $\calF=\{f_1,\dots,f_s\}$ a set of $c$-generic forms in $R=\KK[x_0,\dots,x_n]$.  Suppose that $\projdim_T(I)\le c<\dim(R)=n+1$.  Regarding the homomorphism $\varphi:T\to R$ induced by $z_i\mapsto f_i$ for each $i$, the specialization $\varphi(I)$ has c.i.~quotients. If in addition all forms in $\calF$ are of degree $d$, then $\varphi(I)$ has $d$-c.i.~quotients.
\end{Lemma}

\begin{proof}
    It suffices to point out that 
    \begin{equation}
        \projdim_T(I)=\max\Set{\mu(\braket{\bdz^{\bdbeta_1},\dots,\bdz^{\bdbeta_i}}:\bdz^{\bdbeta_{i+1}}):1\le i\le r-1}
        \label{eqn:projdim-quotients}
    \end{equation}
    where $\mu$ denotes the minimal number of generators. Now, the remaining argument will be similar. We will only encounter the $t\le c\le n$ case, where the Lemma 6.8 and Proposition 6.11 of \cite{arXiv:1907.08172} still apply. 
\end{proof}

With the above preparation, we are ready to state the main result of this section.

\begin{Theorem}
    \label{thm:strongly-generic-ciq}
    Let $R=\KK[x_0,\dots,x_n]$ and $T=\KK[z_1,\dots,z_s]$ be two polynomial rings over the field $\KK$.
    Let $\calF=\{f_1,\dots,f_s\}$ be a set of forms in $R$ and $a,m_1,\dots,m_s$ be positive integers with $a\le m_1+\cdots+m_s$.  
    \begin{enumerate}[a]
        \item \label{thm:strongly-generic-ciq-a}
        If $\calF$ is a strongly generic set of $d$-forms, then the $a$-fold product ideal $I_{a}(f_1^{m_1}\cdots f_s^{m_s})$ has $d$-c.i.~quotients.
        \item If $\projdim_T(I_a(z_1^{m_1}\cdots z_s^{m_s}))\le c\le n$ and $\calF$ is $c$-generic, then the $a$-fold product ideal $I_{a}(f_1^{m_1}\cdots f_s^{m_s})$ has c.i.~quotients. If in addition all forms in $\calF$ are of degree $d$, then $I_{a}(f_1^{m_1}\cdots f_s^{m_s})$ has $d$-c.i.~quotients.  
    \end{enumerate}
\end{Theorem}

\begin{proof}
    The monomial ideal $I_a(z_1^{m_1}\cdots z_s^{m_s})$ in $T$ is known as the ideal of Veronese type. In particular, it is a polymatroidal ideal and has linear quotients; see \cite[Example 12.2.8 and Theorem 12.6.2]{MR2724673}. Now, we can apply \Cref{prop:strongly-generic-ciq} and \Cref{prop:c-generic-ciq}. 
\end{proof}

\begin{Corollary}
    Let $\Sigma:=(\underbrace{f_1,\dots,f_1}_{m_1},\dots,\underbrace{f_s,\dots,f_s}_{m_s})$ be a collection of linear forms in $\KK[x_0,\ldots,x_n]$, with $s, m_1,\dots,m_s\ge 1$, and with $\Supp(\Sigma)=\{f_1,\dots,f_s\}$ being generic. Then for any $1\leq a\le m_1+\cdots+m_s$, the ideal $I_a(f_1^{m_1}f_2^{m_2}\cdots f_s^{m_s})$ has linear quotients.
\end{Corollary}

\begin{proof}
When we are dealing with linear forms, the set $\Supp(\Sigma)$ being strongly generic is equivalent to being generic. Now, we apply \Cref{thm:strongly-generic-ciq} \ref{thm:strongly-generic-ciq-a}.
\end{proof}

Since equi-generated ideals having linear quotients will have linear
resolutions, we recover the linear resolution result in \cite[Theorem
2.3]{arXiv:1906.08346} by the above corollary.

\begin{Corollary}
    Under the assumptions in \Cref{thm:strongly-generic-ciq}, 
    suppose in addition that all forms in $\calF$ are of degree $d$.
    Then the Castelnuovo--Mumford regularity
    \[
        \reg_R(R/I_{a}(f_1^{m_1}\cdots f_s^{m_s})) = d(a+p-1)-p 
    \]
    for $p=\projdim_R(R/I_{a}(f_1^{m_1}\cdots f_s^{m_s}))$.
\end{Corollary}

\begin{proof}
    Note that the ideal $I_a(f_1^{m_1}\cdots f_s^{m_s})$ is equi-generated in degree $ad$. Thus, the largest degree shift of the graded minimal free resolution happens at the top homological degree.
\end{proof}

It is noted in \cite[Remark 6.4 and Corollary 6.6]{arXiv:1907.08172} that any
ideal with $d$-c.i.~quotients has a Koszul stranded Betti table, which is
completely determined by the sizes of minimal generating sets of the successive
colon ideals.  To be more precise, suppose that an ideal $I\subset T=\KK[z_1,\dots,z_s]$ has
$d$-c.i.~quotients with respect to a total order on a generating set
$G(I)=\{g_1>\dots>g_h\}$. For each $1\le i\le h$, let $r_k$ be the
minimal number of generators of $\braket{g_1,\dots,g_{i-1}}:\braket{g_i}$. Now,
$T/I$ has a minimal graded $T$-free resolution $\bdF_\bullet=\bigoplus_i F_i$
with
\[
    F_i=\bigoplus_{k=1}^h T(-(\deg(g_k)+d(i-1)))^{\binom{r_k}{i-1}}
\]
for $i\ge 1$.  

Back to the specialization discussions in \Cref{prop:strongly-generic-ciq} and
\Cref{prop:c-generic-ciq}.  As the sizes of successive colon ideals won't get
larger after specialization by the affiliated proofs, we can compare the Betti numbers of the ideals with ease.

\begin{Proposition}
    Under the assumptions in \Cref{prop:strongly-generic-ciq} or \Cref{prop:c-generic-ciq}, we have
    \[
        \beta_{i,dj}^{R}(R/\varphi(I))\le \beta_{i,j}^{T}(T/I)
    \]
    for any $i\ge 1$.  In particular, for the $a$-fold product ideal $I_a(f_1^{m_1}\cdots f_s^{m_s})=\varphi(I_a(z_1^{m_1}\cdots z_s^{m_s}))$ in $R$, if all forms in $\calF=\{f_1,\dots,f_s\}$ are of degree $d$, then we have
    \begin{equation}
        \beta_{i,d(a+i-1)}^R(R/I_{a}(f_1^{m_1}\cdots f_s^{m_s})) \le
        \beta_{i,a+i-1}^T(T/I_{a}(z_1^{m_1}\cdots z_s^{m_s})) 
        \label{eqn:compare-betti-fold-product}
    \end{equation}
    and
    \[
        \beta_{i,j}^R(R/I_{a}(f_1^{m_1}\cdots f_s^{m_s}))=0 \qquad\text{for $j\ne d(a+i-1)$}
    \]
    for each $i\ge 1$.
\end{Proposition}

\begin{proof}
    The two quotient modules $R/\varphi(I)$ and $T/I$ have graded minimal free resolutions $\calF_\bullet=\bigoplus_i F_i$ and $\calF_\bullet'=\bigoplus_i F_i'$ respectively such that for each direct summand 
    \[
        R(-d(\deg(\bdz^{\bdbeta_k})+(i-1)))^{\binom{r_k}{i-1}}
    \]
    of $F_i$, there is a corresponding direct summand 
    \[
        T(-(\deg(\bdz^{\bdbeta_k})+(i-1)))^{\binom{r_k'}{i-1}}
    \]
    of $F_i'$. Furthermore, $r_k\le r_k'$ for each $k$. The rest of the argument is clear.
\end{proof}

\begin{Corollary}
    \label{cor:compare-projdim}
    Under the assumptions in \Cref{prop:strongly-generic-ciq} or \Cref{prop:c-generic-ciq}, we have
    \[
        \projdim_R(R/\varphi(I))\le \projdim_T(T/I).
    \]
\end{Corollary}

Conversely, equality in \eqref{eqn:compare-betti-fold-product} for $i=2$ will
imply that $\calF$ actually forms a regular sequence in the zero-dimensional case.

\begin{Corollary}
    Under the assumptions in \Cref{prop:strongly-generic-ciq}, suppose in
    addition that $I\subset T$ is zero-dimensional and equi-generated with 
    $\beta_{2,d(a+1)}^R(R/\varphi(I)) = \beta_{2,a+1}^T(T/I)$, then
    $f_1,\dots,f_s$ is indeed a regular sequence.
\end{Corollary}

\begin{proof}
    Suppose that in the proof of \Cref{prop:strongly-generic-ciq}, $\braket{\bdz^{\bdbeta_1},\dots,\bdz^{\bdbeta_i}}:\bdz^{\bdbeta_{i+1}}$ is minimally generated by $r_i$ variables for $i=0,1,\dots,r-1$.  Meanwhile, suppose that $\braket{\varphi(\bdz^{\bdbeta_1}), \dots, \varphi(\bdz^{\bdbeta_i})} : \varphi(\bdz^{\bdbeta_{i+1}})$ is minimally generated by $r_i'$ forms in $\calF$; when this colon ideal is $R$, we will choose $r_i'=0$.  We have mentioned earlier that $r_i\ge r_i'$ for all $i$. It is clear that the two total Betti numbers in the condition are $\sum_{i}r_i$ and $\sum_i r_i'$ respectively.  Thus, the equality of these two Betti numbers implies that $r_i=r_i'$ for all $i$. But as $T/I$ is zero-dimensional, $\projdim_T(T/I)=\dim(T)$. Thus, by applying \eqref{eqn:projdim-quotients}, one has
    \[
        r_{i_0}=\dim(T)\ge \dim(R)\ge r_{i_0}' 
    \]
    for some $i_0$. Therefore, by our strongly generic assumption, this implies that $n+1=\dim(R)=\dim(T)=s$ and $f_1,\dots,f_s$ indeed form a regular sequence.
\end{proof}

\section{Symbolic powers of saturated uniform $a$-fold product ideal}

Inspired by the above work on the graded minimal free resolution of
$I_a(f_1^{m_1}\cdots f_s^{m_s})$, it is natural to seek a detailed description
of the Betti table of $I_{a}(z_1^{m_1}\cdots z_s^{m_s})$ in
$T=\KK[z_1,\dots,z_s]$. A starting point will be treating the uniform case when
$m_1=\cdots=m_s=b$. This case can be handled with ease by the results from
\cite{arXiv:1907.04288}. Indeed, more information can be retrieved. We will be able to scrutinize the projective dimension of this ideal, which will in turn allow us to describe its symbolic powers in some nice cases, via the decomposition work in Section 2. After that, we will study some asymptotic quantities related to the symbolic powers in these cases.

But first, we have to recall some notations from \cite{arXiv:1907.04288}. A sequence
$\bdlambda=(\lambda_1,\dots,\lambda_s)$ of non-negative integers is called a
\emph{partition} of $d$ of length $s$, if $\lambda_1\le \cdots \le \lambda_s$
and $|\bdlambda|\coloneqq \lambda_1+\cdots+\lambda_s=d$. The symmetric group $\frakS_s$ acts on $T=\KK[z_1,\dots,z_s]$ by permutations of the variables. Let $I\subset T$ be an $\frakS_s$-fixed monomial ideal. Set 
\[
    \calP(I)\coloneqq\Set{\bdlambda\text{ is a partition of length $s$}: \bdz^{\bdlambda}\in I} 
\]
and 
\[
    \Lambda(I)\coloneqq\Set{\bdlambda:\bdz^{\bdlambda} \text{ is a minimal monomial generator of $I$}}.
\]
It is clear that $\Lambda(I)$ is the set of minimal elements in $\calP(I)$ with respect to the partial ordering of componentwise comparison.

The ideal $I$ is called \emph{symmetric
    shifted} if for every $\bdlambda=(\lambda_1,\dots,\lambda_s)\in \calP(I)$
(or equivalently, $\bdlambda\in \Lambda(I)$) and $1\le k<s$ with $\lambda_k<\lambda_s$,
one has $\bdz^{\bdlambda}z_k/z_s \in I$.

Given any partition $\bdlambda=(\lambda_1,\dots,\lambda_s)$ of $d$, one defines
\[
    p(\bdlambda)\coloneqq|\Set{k:\lambda_k<\lambda_s-1}| \quad \text{and}\quad
    r(\bdlambda)\coloneqq\left| \Set{k:\lambda_k=\lambda_s} \right|.
\]
Besides, its \emph{truncation} is $\bdlambda_{\le k}\coloneqq(\lambda_1,\dots,\lambda_k)$,
and its \emph{type} is $\type(\bdlambda)\coloneqq (t_0,\dots,t_d)$ where
$t_i=\left| \Set{k:\lambda_k=i} \right|$. Furthermore, we set
$\type(\bdlambda)!\coloneqq \prod_{i} t_i!$.

The key result that we shall apply repeatedly is the following.

\begin{Lemma}
    [{\cite[Theorem 3.2, Corollary 5.6]{arXiv:1907.04288}}]
    \label{lem:ss-formula}
    If $I\subset T$ is symmetric shifted, then $I$ has linear quotients and
    \begin{equation}
        \scalebox{0.9}{%
        $\displaystyle\beta_{i,i+d}(I)=\sum_{\bdlambda\in\Lambda(I),|\bdlambda|=d}\left(
            \sum_{k+l=i}\frac{p(\bdlambda)!}{\type(\bdlambda_{\le p(\bdlambda)})!}
            \binom{s}{p(\bdlambda)} \binom{s-p(\bdlambda)}{r(\bdlambda)+k}
            \binom{r(\bdlambda)+k-1}{k}\binom{p(\bdlambda)}{l}\right).$}
        \label{eqn:betti-formula}
    \end{equation}
\end{Lemma}

Back to our uniform $a$-fold product ideal $I=I_a(z_1^b\cdots z_s^b)$ in $T=\KK[z_1,\dots,z_s]$.
Then 
\[
\Lambda(I)=\Set{\bdlambda=(\lambda_1,\dots,\lambda_{s}): \text{$\bdlambda$ is a partition of $a$ with $\lambda_{s}\le b$}}.
\]
One can check with ease that $I$ is {symmetric shifted}.  Consequently, its Betti table is clear by the previous lemma.

\begin{Example}
    Consider the ideal $I=I_a(z_1^2\cdots z_s^2)$ in $T=\KK[z_1,\dots,z_s]$.  This ideal is symmetric shifted and $\Lambda(I)$ consists of partitions 
    \[
    \bdlambda_{u,v}\coloneqq (0^{s-u-v},1^u,2^v)=(\underbrace{0,\dots,0}_{s-u-v},\underbrace{1,\dots,1}_{u},\underbrace{2,\dots,2}_{v})
    \] 
    with $u+2v=a$ and $v\ge 1$, as well as the partition 
    \[
    \bdlambda_{a,0}\coloneqq (0^{s-a},1^a)=(\underbrace{0,\dots,0}_{s-a},\underbrace{1,\dots,1}_{a}).
    \]
    It is easy to see that 
    \[
    p(\bdlambda_{a,0})=0 \text{ and } r(\bdlambda_{a,0})=a.
    \]
    And if $v>0$, then
    \[
        p(\bdlambda_{u,v})=s-u-v \text{ and } r(\bdlambda_{u,v})=v.
    \]
    Therefore, the formula \eqref{eqn:betti-formula} says
    \begin{align*}
        \beta_{i,i+a}(I)&=
        \sum_{k+l=i}\frac{0!}{0!}\binom{s}{0}\binom{s-0}{a+k}\binom{a+k-1}{k}\binom{0}{l} \\
        &+ \sum_{v=1}^{\floor{a/2}}\sum_{k+l=i}\frac{(s+v-a)!}{(s+v-a)!}\binom{s}{s+v-a}\binom{a-v}{v+k}\binom{v+k-1}{k}\binom{s+v-a}{l}\\
        &= \binom{s}{a+i}\binom{a+i-1}{i}+\sum_{v=1}^{\floor{a/2}}\sum_{k+l=i}\binom{s}{s+v-a}\binom{a-v}{v+k}\binom{v+k-1}{k}\binom{s+v-a}{l}.
    \end{align*}
    For instance, when $a=3$ and $s=4$, the above formula says that
    \[
        \beta_{0,3}(I)=16, \quad
        \beta_{1,4}(I)=33, \quad
        \beta_{2,5}(I)=24, \quad
        \beta_{3,6}(I)=6, \quad
        \beta_{4,7}(I)=0.
    \]
    It agrees with the following Betti table given by \texttt{Macaulay2}
    \cite{M2}:
    \begin{verbatim}
               0  1  2  3 4
        total: 1 16 33 24 6
            0: 1  .  .  . .
            1: .  .  .  . .
            2: . 16 33 24 6
    \end{verbatim}
\end{Example}

\subsection{Projective dimension and primary-type decomposition}
The ultimate aim of this subsection is to give a clean formula for the symbolic powers of the ideal of generalized star configuration $I_a(\calF^b)=I_a(f_1^{b}\cdots f_s^b)$ in some nice cases. To achieve that, we start with investigating the projective dimension in the monomial case. 

\begin{Proposition}
    \label{prop:projdim-monomial}
    Consider the uniform $a$-fold product ideal $I=I_a(z_1^b\cdots z_s^b)$ in
    $T=\KK[z_1,\dots,z_s]$ with $b\le a\le bs$. 
    \begin{enumerate}[i]
        \item The ideal $I$ is saturated if and only if $a> (b-1)s+1$. 
        \item If $a> (b-1)s+1$, then $\projdim_T(T/I)=s-\delta$ for $\delta \coloneqq a- (b-1)s-1$.
    \end{enumerate}
\end{Proposition}

\begin{proof}
    Depending on the parameter $a$, we have two cases.
    \begin{enumerate}[a]
        \item If $b\le a\le s(b-1)+1$, we can always find some $\bdlambda_0=(\lambda_1,\dots,\lambda_s)\in \Lambda(I)$ such that $\lambda_{s-1}<\lambda_{s}=b$. This in particular means $r(\bdlambda_0)=1$.  Now, we consider the formula \eqref{eqn:betti-formula} with $i=s-1$, $\bdlambda=\bdlambda_0$, $k=s-p(\bdlambda_0)-1$ and $l=p(\bdlambda_0)$. This combination of setting contributes a positive value to the summation. Hence $\beta_{s-1,s-1+a}(I)\ne 0$ and $\projdim(T/I)=s=\dim(T)$. This implies that $\depth(T/I)=0$ and $I$ is not saturated.

        \item If $a>(b-1)s+1$, then for each fixed $\bdlambda=(\lambda_1,\dots,\lambda_s)\in \Lambda(I)$, we will have $\lambda_s=b$. As $|\bdlambda|=a=(b-1)s+1+\delta$, $r(\bdlambda)\ge 1+\delta$. Now for $i\ge s-\delta$, 
        \[
        (r(\bdlambda)+k)+l=r(\bdlambda)+i\ge s+1>(s-p(\bdlambda))+p(\bdlambda)=s,
        \]
         giving
            \[
                \binom{s-p(\bdlambda)}{r(\bdlambda)+k} \binom{p(\bdlambda)}{l}=0
            \]
            in the formula \eqref{eqn:betti-formula}. This implies that the total Betti number $\beta_{i}(I)=0$ for $i\ge s-\delta$.

            Meanwhile, we can find the partition $\bdlambda_0=((b-1)^{s-1-\delta}, b^{1+\delta})\in \Lambda(I)$ with $r(\bdlambda_0)=\delta+1$.  An argument as above will show that $\beta_{s-\delta-1}(I)\ne 0$.  Therefore, $\projdim(I)=s-\delta-1$, or equivalently, $\projdim(T/I)= s-\delta$. Whence, $I$ is saturated.
            \qedhere
    \end{enumerate}
\end{proof}

\begin{Proposition}
    \label{prop:symbolic-power-strongly-generic}
    With the assumptions in \Cref{setting-uniform-fold-product}, we further assume that $a>(b-1)s+1$ and write $\delta\coloneqq a-(b-1)s-1$. If $\delta\ge s-n$ and $\calF=\{f_1,\dots,f_s\}$ is a set of $(s-\delta)$-generic $d$-forms in $R=\KK[x_0,\dots,x_n]$, then
    \[
        I_a(\calF^b)=
        \bigcap_{c=c_0}^{s-\delta} \left( 
        \bigcap_{1\le i_1<\cdots < i_c\le s} \braket{f_{i_1},\dots,f_{i_c}}^{\mu_a^0+b(c-c_0)}
        \right).
    \]
\end{Proposition}

\begin{proof}
    Consider as well $\calF_{\bdz}=\{z_1,\dots,z_s\}$ in $T=\KK[z_1,\dots,z_s]$. We have seen in the proof of \Cref{thm:strongly-generic-ciq} that $I_a(\calF_{\bdz}^{b})$ has linear quotients.  It follows from \Cref{cor:compare-projdim} and \Cref{prop:projdim-monomial} that
    \[
        \projdim_R(R/I_a(\calF^b))\le \projdim_T(T/I_a(\calF_{\bdz}^b))=s-\delta\le n.
    \]
    Therefore, we can apply the fact in \Cref{rmk:ass-cond} and obtain the decomposition in \Cref{prop:dec-c-generic} with $\hat{c}=s-\delta$. 
    
    In that decomposition, the exponent $\mu_a(i_1,\dots,i_c)=\mu_a^0+b(c-c_0)$ for $c_0\le c\le s-\delta$ by the direct computation at the end of Section 2. And when $c<c_0$, the exponent $\mu_a(i_1,\dots,i_c)<0$. Obviously, we can remove these redundant components from the decomposition.
\end{proof}

\begin{Remark}
    \label{rmk:ass-set}
    In order to study the symbolic powers of $I_a(\calF^b)$ later, we still need to determine its associate primes via \Cref{prop:symbolic-power-strongly-generic}. By the decomposition formula there, it is clear that $\Ass(R/I_a(\calF^b))$ is a subset of the disjoint union 
    \[
        \calA\coloneqq \bigsqcup_{\substack{c_0\le c\le s-\delta \\ 1\le i_1<\cdots<i_c\le s}} \Ass(R/\braket{f_{i_1},\dots,f_{i_c}}),
    \]
    since the forms in $\calF$ are locally complete intersections. On the other hand, let us take arbitrary $\frakp\in \calA$. Say, $\frakp\in \Ass(R/\braket{f_1,\dots,f_c})$. Then $I_a(\calF^b)R_{\frakp}=I_{a-b(s-c)}(f_1^b\cdots f_c^b)R_{\frakp}$ by our previous calculation in the equation \eqref{eqn:local-product}.
    Notice that $\dim(R_{\frakp})=c$ and the images of $f_1,\dots,f_c$ in $R_{\frakp}$ form a regular sequence. Since
    \[
    a-b(s-c)\le (b-1)c+1
    \]
    for $c_0\le c\le s-\delta$, it follows from \Cref{prop:projdim-monomial} that $I_{a-b(s-c)}(f_1^b\cdots f_c^b)R_{\frakp}$ is not saturated in $R_{\frakp}$. Therefore, $\frakp R_{\frakp}$ is an associated prime of $R_{\frakp}/I_{a-b(s-c)}(f_1^b\cdots f_c^b)R_{\frakp}$ and equivalently $\frakp$ is an associated prime of $R/I_a(\calF^b)$. In short, $\calA=\Ass(R/I_a(\calF^b))$.

    Consequently, the decomposition in \Cref{prop:symbolic-power-strongly-generic} is irredundant in the obvious sense.
\end{Remark}

Now, we are ready to state the first main result of this section.

\begin{Theorem}
    \label{prop:decomposition-strongly-generic-special}
    With the assumptions in \Cref{prop:symbolic-power-strongly-generic}, the symbolic powers of $I_{a}(\calF^b)$ can be written as
    \begin{align*}
        I_{a}(\calF^b)^{(m)}&=\bigcap_{c=c_0}^{s-\delta} 
        \left( 
            \bigcap_{1\le i_1<\cdots < i_c\le s} \braket{f_{i_1},\dots,f_{i_c}}^{m(\mu_a^0+b(c-c_0))}
        \right)\\
        &=\bigcap_{c=c_0}^{s-\delta} I_{c,\calF}^{(m(\mu_a^0+b(c-c_0)))}.
    \end{align*}
\end{Theorem}

\begin{proof}
    We will adopt the symbol $\calA$ in the \Cref{rmk:ass-set}.  By definition,
    \begin{align*}
        I_{a}(\calF^b)^{(m)}&=
        \bigcap_{\frakp\in \calA}\left(I_a(\calF^b)^m R_{\frakp} \cap R\right).
    \end{align*}
    For each $\frakp\in \calA$ with $\Ht(\frakp)=c$, we can find, for instance, $f_1,\dots,f_c\in \frakp$ while $f_{c+1},\dots,f_s\notin \frakp$. Whence, 
    \begin{align*}
        I_a(\calF^b)^m R_{\frakp}&=I_{a-b(s-c)}(f_1^{b}\cdots f_c^{b})^mR_{\frakp}\\
        &=I_{ma-mb(s-c)}(f_1^{mb}\cdots f_c^{mb})R_{\frakp}\\ 
        &=\bigcap_{c_0\le c'\le c} \left(\bigcap_{1\le i_1<\cdots<i_{c'}\le c}\braket{f_{i_1},\dots,f_{i_{c'}}}^{m(\mu_a^0+b(c'-c_0))}R_{\frakp}\right)\\
        &=\bigcap_{c_0\le c'\le s-\delta} \left(\bigcap_{1\le i_1<\cdots<i_{c'}\le s}\braket{f_{i_1},\dots,f_{i_{c'}}}^{m(\mu_a^0+b(c'-c_0))}R_{\frakp}\right).
    \end{align*}
    The first three equalities are due to formulas  \eqref{eqn:local-product}, \eqref{eqn-power} and \Cref{lem:reg-case} respectively. The last one is also clear, since any $f_j$ not in $\frakp$ will lead to the localization ideal being $R_{\frakp}$. And this induces
    \[
        I_{a}(\calF^b)^{(m)}=\bigcap_{\frakp\in\calA}\left(\bigcap_{c=c_0}^{s-\delta} \left( \bigcap_{1\le i_1<\cdots < i_c\le s} \braket{f_{i_1},\dots,f_{i_c}}^{m(\mu_a^0+b(c-c_0))}R_{\frakp}\cap R \right)\right).
    \]

    At the same time, for each $c$ with $c_0\le c\le s-\delta$, we have
    \begin{align*}
        I_{c,\calF}^{(m(\mu_a^0+b(c-c_0)))}&=\bigcap_{1\le i_1<\cdots < i_c\le s} \braket{f_{i_1},\dots,f_{i_c}}^{m(\mu_a^0+b(c-c_0))}\\
        &\subseteq \bigcap_{\frakp\in \calA}\left( \bigcap_{1\le i_1<\cdots < i_c\le s} \braket{f_{i_1},\dots,f_{i_c}}^{m(\mu_a^0+b(c-c_0))}R_{\frakp} \cap R \right) \\
        &\subseteq \bigcap_{\substack{\frakp\in \calA\\ \Ht(\frakp)=c}}\left( \bigcap_{1\le i_1<\cdots < i_c\le s} \braket{f_{i_1},\dots,f_{i_c}}^{m(\mu_a^0+b(c-c_0))}R_{\frakp} \cap R \right) \\
        &= I_{c,\calF}^{(m(\mu_a^0+b(c-c_0)))}.
    \end{align*}
    Here, we need \cite[Theorem 3.6(1)]{MR3683102} for the two equalities. 

    After putting together every piece of information, we arrive at the expected formulas.   
\end{proof}

\subsection{Monomial case}
To study the symbolic powers of the uniform $a$-fold product ideal $I_a(\calF^b)$, we have to check with its monomial prototype $I_a(z_1^b\cdots z_s^b)\subset T=\KK[z_1,\dots,z_s]$ first. So, temporarily, we shift our focus to the monomial case.  As the initial step, we can strengthen the last piece of \Cref{prop:projdim-monomial} as follows. 

\begin{Proposition}
    \label{cor:projdim-symbolic-power}
    With the assumptions in \Cref{setting-uniform-fold-product}, we consider the set $\calF_{\bdz}=\{z_1,\dots,z_s\}$ in the polynomial ring $\KK[z_1,\dots,z_s]$. Suppose that $(b-1)s+1<a\le bs$ and we write $\delta\coloneqq a-(b-1)s-1$.  Then, for each positive integer $m$, $\projdim_T(T/I_a(\calF_{\bdz}^b)^{(m)})=s-\delta$ and $\reg_T(T/I_a(\calF_{\bdz}^b)^{(m)})= m(\mu_a^0+b(s-c_0))-1$.
\end{Proposition}

\begin{proof}
    Firstly, in order to apply \Cref{lem:ss-formula}, we need to verify that $I=I_a(\calF_{\bdz}^b)^{(m)}$ is symmetric shifted. Notice that this ideal is obviously $\frakS_{s}$-invariant. Thus, we will take arbitrary partition ${\bdlambda}=(\lambda_1,\dots,\lambda_s)\in \calP(I_a(\calF_{\bdz}^b)^{(m)})$.  It follows from \Cref{prop:decomposition-strongly-generic-special} that this is equivalent to saying that $\bdz^{\bdlambda}\in I_c\coloneqq I_{c,\calF_{\bdz}}^{(m(\mu_a^0+b(c-c_0)))}$ for each $c=c_0,c_0+1,\dots,s-\delta$.  Notice that $I_c$ is symmetric shifted by \cite[Theorem 4.3]{arXiv:1907.04288} for each $c$. Therefore, if $\lambda_k<\lambda_s$, then $\bdz^{\bdlambda}z_k/z_s\in I_c$ for each such $c$. Consequently, $\bdz^{\bdlambda}z_k/z_s\in I=\bigcap_c I_c$, meaning $I$ is symmetric shifted.

    Furthermore, by the above argument, a partition ${\bdlambda}\in \calP(I_a(\calF_{\bdz}^b)^{(m)})$ if and only if $\bdz^{\bdlambda}\in I_c$ for $c_0\le c\le s-\delta$, and by \cite[Proposition 4.1]{arXiv:1907.04288}, if and only if $|\bdlambda_{\le c}|\ge m(\mu_a^0+b(c-c_0))$ for every such $c$. Thus, each $\bdlambda\in \Lambda(I_a(\calF_{\bdz}^b)^{(m)})$ satisfies $\lambda_{s-\delta}=\cdots=\lambda_s$. In particular, $r(\bdlambda)\ge \delta+1$, which implies that $\projdim_T(I_a(\calF_{\bdz}^b)^{(m)})\le s-\delta-1$ by studying the formula \eqref{eqn:betti-formula}.  Meanwhile, we do find $\bdlambda_0=((m(b-1))^{s-(\delta+1)},(mb)^{\delta+1})\in \Lambda(I_a(\calF_{\bdz}^b)^{(m)})$ with $r(\bdlambda_0)=\delta+1$. This means that $\projdim_T(I_a(\calF_{\bdz}^b)^{(m)})= s-\delta-1$, and equivalently $\projdim_T(T/I_a(\calF_{\bdz}^b)^{(m)})= s-\delta$.

    To determine the Castelnuovo--Mumford regularity, it suffices to find 
    \[
    \max\{|\bdlambda|:\bdlambda\in \Lambda(I_a(\calF_{\bdz}^b)^{(m)})\} 
    \]
    by \Cref{lem:ss-formula}. Since $c_0\le s-\delta$, we have two cases.
    \begin{enumerate}[a]
        \item Suppose that $c_0=s-\delta$. It is clear that $\bdlambda\in \Lambda(I_a(\calF_{\bdz}^b)^{(m)})$ if and only if $|\bdlambda_{\le c_0}|=m\mu_a^0$ with $\lambda_{c_0}=\cdots=\lambda_s$. Thus, the maximum is $m\mu_a^0(s-c_0+1)$, achieved at $(0^{c_0-1},(m\mu_a^0)^{s-c_0+1})$.
        \item Suppose that $c_0<s-\delta$ and take arbitrary $\bdlambda\in \Lambda(I_a(\calF_{\bdz}^b)^{(m)})$. 
            
            We claim first that $\lambda_{s-\delta}\le mb$. Suppose for contradiction that $\lambda_{s-\delta}> mb$. Let $t_0\coloneqq \min\Set{c_0\le t\le s-\delta:\lambda_t>mb}$. If $t_0=c_0$, then $\bdlambda'=(0^{c_0-1}, (mb)^{s-c_0+1})\in \calP(I_a(\calF_{\bdz}^b)^{(m)}$. If $t_0>c_0$, then $\left| \bdlambda_{\le t_0-1} \right|\ge m(\mu_a^0+b(t_0-1-c_0))$. As $\lambda_c>mb$ for $c\ge t_0$, we also have $\left| \bdlambda_{\le c}'' \right|\ge m(\mu_a^0+b(c-c_0))$  for $c\ge t_0$ and $\bdlambda''= (\lambda_1,\dots,\lambda_{c-1},mb,\dots,mb)$. This implies that $\bdlambda''\in \calP(I_a(\calF_{\bdz}^b)^{(m)})$.  In both subcases, we have a contradiction to the minimality of $\bdlambda$. 
            Thus, $\lambda_{s-\delta}\le mb$. 
            
            We next claim that $\left| \bdlambda_{\le s-\delta} \right|=m(\mu_a^0+b(s-\delta-c_0))$ in this situation.  Suppose that this is not true. It follows that $\left| \bdlambda_{\le s-\delta} \right|> m(\mu_a^0+b(s-\delta-c_0))$. As $\mu_a^0\le b$, we will additionally have $\left| \bdlambda_{\le t} \right|>m(\mu_a^0+b(t-c_0))$ for $c_0\le t\le s-\delta$. Now, suppose that $t_1=\min\Set{1\le t\le s: \lambda_t>0}$. If we take $\bdlambda'''=(0,\dots,0,\lambda_{t_1}-1,\lambda_{t_1+1},\dots,\lambda_{s})$, then $\left| \bdlambda_{\le t}''' \right|\ge m(\mu_a^0+b(t-c_0))$ for $c_0\le t\le s-\delta$. This implies that $\bdlambda'''\in \calP(I_a(\calF_{\bdz}^b)^{(m)})$, contradicting the minimality of $\bdlambda$.

            Now, as $\left| \bdlambda_{\le s-\delta} \right|=m(\mu_a^0+b(s-\delta-c_0))$ and $\lambda_{s-\delta}=\cdots=\lambda_s\le mb$, the maximum of $\left| \bdlambda \right|$ is $m(\mu_a^0+b(s-c_0))$, achieved at $(0^{c_0-1},m\mu_a^0,(mb)^{s-c_0})$.
    \end{enumerate}
    Notice that $c_0=s-\delta$ precisely when $s-\delta=\ceil{\frac{s-\delta}{b}}$, and precisely when $b=1$ or $\delta=s-1$. Whence, the two maxima computed above all agree with $ma$. Therefore, we can simply take the second format and obtain $\reg_T(I_a(\calF_{\bdz}^b)^{(m)})= m(\mu_a^0+b(s-c_0))$ by \Cref{lem:ss-formula}.
\end{proof}

Next, we are going to study the resurgence of the generalized star configuration in the monomial case. Recall that if $I$ is a nonzero graded ideal in the standard graded ring $R$, the \emph{resurgence} of $I$ is defined to be
\begin{equation*}
    \rho(I)\coloneqq \sup\Set{\frac{m}{r}: I^{(m)}\nsubseteq I^r}.
\end{equation*}
We will always denote by $\alpha(I)$ the least degree of nonzero forms in $I$. Meanwhile, the \emph{Waldschmidt constant} $\widehat{\alpha}(I)$ of $I$ is defined to be 
\[
    \widehat{\alpha}(I)\coloneqq\lim_{m\to \infty} \frac{\alpha(I^{(m)})}{m}.
\]
This limit is known to exist and satisfies
\[
    \frac{\alpha(I)}{\widehat{\alpha}(I)}\le \rho(I)
\]
by \cite[Lemma 2.3.1 and Theorem 1.2.1]{MR2629595}.

\begin{Proposition}
    \label{thm:rho-monomial}
    With the assumptions in \Cref{cor:projdim-symbolic-power}, for the uniform $a$-fold product ideal $I=I_{a}(\calF_{\bdz}^b)$, we have
    \[
        \rho(I)= \frac{\alpha(I)}{\widehat{\alpha}(I)}=\frac{a(s-\delta)}{s(\mu_a^0+b(s-\delta-c_0))}.
    \]
\end{Proposition}

\begin{proof}
    Partitions $\bdlambda\in\calP(I_a(\calF_{\bdz}^b)^{(m)})$ are characterized by the requirements 
    \[
        |\bdlambda_{\le c}|\ge m(\mu_a^0+b(c-c_0))\quad \text{for $c_0\le c\le s-\delta$}.
    \]
    It is clear that $\lambda_c\ge \ceil{\frac{m(\mu_a^0+b(c-c_0))}{c}}$ for each such $c$. Furthermore, since $\mu_a^0\le b$ and $c_0\le s-\delta$, one has ${\frac{m\mu_a^0}{c_0}}\le {\frac{m(\mu_a^0+b(s-\delta-c_0))}{c_0+(s-\delta-c_0)}}$ and consequently $\ceil{\frac{m\mu_a^0}{c_0}}\le \ceil{\frac{m(\mu_a^0+b(s-\delta-c_0))}{s-\delta}}$.  Meanwhile, $\bdlambda\in\calP(I_a(\calF_{\bdz}^b)^r) = \calP(I_{ar}(\calF_{\bdz}^{br}))$ if and only if $\sum_i \min(\lambda_i,br)\ge ar$.  Now, for the containment $I_a(\calF_{\bdz}^b)^{(m)}\subseteq I_a(\calF_{\bdz}^b)^r$ with $m\ge r\ge 1$, we have three cases. 
    \begin{enumerate}[a]
        \item Suppose that $\lambda_{c_0}\ge br$. Then, $\lambda_i\ge br$ for $c_0\le i\le s$. Whence, the minimum of $\sum_{i=1}^s \min(\lambda_i,br)$ in this case is exactly $br(s-c_0+1)$. Now, the requirement for the containment is 
            \[
                br(s-c_0+1)\ge ar,
            \]
            which holds automatically by \eqref{eqn:mua0} and \eqref{eqn:a}.  This means that the requirement for the containment is void in this case.
        \item 
            Suppose that $\lambda_{s-\delta}< br$. Then 
            \[
                \min(\lambda_i,br)\ge \lambda_{s-\delta}\ge \ceil{\frac{m(\mu_a^0+b(s-\delta-c_0))}{s-\delta}}
            \]
            for $s-\delta+1\le i\le s$. Therefore, the minimum of $\sum_{i=1}^s \min(\lambda_i,br)$ in this case coincides with
            \begin{align*}
                A\coloneqq m(\mu_a^0+b(s-\delta-c_0))+\delta \ceil{\frac{m(\mu_a^0+b(s-\delta-c_0))}{s-\delta}}.
            \end{align*} 
            Whence, the condition for the containment is simply $A \ge ar$.  If we write 
            \begin{equation}
                m(\mu_a^0+b(s-\delta-c_0))=q_0(s-\delta)+q_1 \qquad \text{with $1\le q_1\le s-\delta$,}
                \label{eqn:mmuab}
            \end{equation}
            then the requirement $A\ge ar$ is equivalent to saying
            \[
                r\le \frac{q_0s+q_1+\delta}{a}.
            \]
        \item Suppose that $\lambda_{c_0}<br\le \lambda_{s-\delta}$. Then let $c$ be the smallest such that $\lambda_c\ge br$. It is clear that $c_0<c\le s-\delta$. Similar to the above discussions, we find the minimum of $\sum_{i=1}^s \min(\lambda_i,br)$ in this case agreeing with
            \[
                B\coloneqq m(\mu_a^0+b(c-c_0-1))+(s-c+1)br.
            \]
            Thus, the condition for the containment is simply $B\ge ar$, or equivalently $r\le m$ by \eqref{eqn:mua0} and \eqref{eqn:a}. Since $r\le m$ is always true for $I_a(\calF_{\bdz}^b)^{(m)}\subseteq I_a(\calF_{\bdz}^b)^r$, the requirement for the containment is void in this case.  
    \end{enumerate}
    To sum up, $I_a(\calF_{\bdz}^b)^{(m)}\nsubseteq I_a(\calF_{\bdz}^b)^r$ if and only if
    \[
        \min\Set{\lambda_{s-\delta}:\bdlambda\in\calP(I_a(\calF_{\bdz}^b)^{(m)})}=\ceil{\frac{m(\mu_a^0+b(s-\delta-c_0))}{s-\delta}}< br
    \]
    while
    $r>\frac{q_0s+q_1+\delta}{a}$. Whence,
    \[
        \frac{m}{r}<\frac{ma}{q_0s+q_1+\delta}\le \frac{a(s-\delta)}{s(\mu_a^0+b(s-\delta-c_0))}.
    \]
    The second inequality can be verified directly by paying attention to the assumptions in \eqref{eqn:mmuab}. Consequently, we have established 
    \[
        \rho(I)\le \frac{a(s-\delta)}{s(\mu_a^0+b(s-\delta-c_0))}.
    \]

    On the other hand, it is clear that $\alpha(I)=a$, while
    \begin{align}
        \alpha(I^{(m)})&\ge \max\Set{
            \alpha(I_{c,\calF_{\bdz}}^{(m(\mu_a^0+b(c-c_0)))}): c_0\le c\le s-\delta
        } \label{eqn:Im-min-1}\\
        &\ge \alpha(I_{s-\delta,\calF_{\bdz}}^{(m(\mu_a^0+b(s-\delta-c_0)))}) \label{eqn:Im-min-2}\\
        &=m(\mu_a^0+b(s-\delta-c_0))+\ceil{\frac{m(\mu_a^0+b(s-\delta-c_0))}{s-\delta}}\delta 
        \notag
    \end{align}
    by \Cref{prop:decomposition-strongly-generic-special} and \cite[Proposition 4.1]{arXiv:1907.04288}.  Notice that we can find a (unique) partition 
    \[
        \bdlambda_0=(\lambda_1,\dots,\lambda_s)\in \Lambda(I_{s-\delta,\calF_{\bdz}}^{(m(\mu_a^0+b(s-\delta-c_0)))}) 
    \]
    with 
    \[
        \lambda_1\le \cdots\le \lambda_s=\ceil{\frac{m(\mu_a^0+b(s-\delta-c_0))}{s-\delta}}\le \lambda_1+1.
    \]
    Obviously we have $\bdlambda_0\in \calP(I_{c,\calF_{\bdz}}^{(m(\mu_a^0+b(c-c_0)))})$ for $c_0\le c\le s-\delta$. The existence of such $\bdlambda_0$ implies that the comparisons in \eqref{eqn:Im-min-1} and \eqref{eqn:Im-min-2} are indeed equalities.  Therefore, 
    \begin{align*}
        \rho(I)&\ge \frac{\alpha(I)}{\widehat{\alpha}(I)}={a}\left/{\lim\limits_{m\to \infty} \dfrac{m(\mu_a^0+b(s-\delta-c_0))+\ceil{\frac{m(\mu_a^0+b(s-\delta-c_0))}{s-\delta}}\delta}{m}}\right.\\
        &= \frac{a(s-\delta)}{s(\mu_a^0+b(s-\delta-c_0))}.
    \end{align*}
    And this completes the proof.
\end{proof}

Here is some information regarding the symbolic defect in the monomial
case.  

\begin{Proposition}
    \label{thm:symbolic-defect-monomial}
    With the assumptions in \Cref{cor:projdim-symbolic-power},  for the
    uniform $a$-fold product ideal $I=I_{a}(\calF_{\bdz}^b)$ and the positive
    integer $m$, we assign  
    \[
        \Lambda(I,m)\coloneqq \{ \bdlambda:  |\bdlambda_{\le s-\delta}|= m(\mu_a^0+b(s-\delta-c_0))\text{ and }\lambda_{s-\delta}=\lambda_{s-\delta+1}=\cdots =\lambda_s<mb\}.
    \]
    Now, the symbolic defect $\sdefect(I,m)$ is given by
    $\sum_{\bdlambda\in\Lambda(I,m)}\frac{s!}{\type(\bdlambda)!}$.
\end{Proposition}

\begin{proof}
    By reading the previous proof with $r=m$, we acknowledge that any partition $\bdlambda\in\Lambda(I_a(\calF_{\bdz}^b)^{(m)})$ with $\bdz^{\bdlambda}\notin I_a(\calF_{\bdz}^b)^{m}$ satisfies the requirements that
    \begin{equation}
        |\bdlambda_{\le c}|\ge m(\mu_a^0+b(c-c_0))\label{eqn:15}
    \end{equation}
    for all $c_0\le c\le s-\delta$, and
    \begin{equation*}
        \quad \lambda_{s-\delta}=\lambda_{s-\delta+1}=\cdots =\lambda_s<mb. 
    \end{equation*}
    Since $\lambda_c\le \lambda_{s-\delta}\le mb-1$ for $c_0\le c\le s-\delta$. It follows that
    \begin{align*}
        \left| \bdlambda_{\le c} \right|=\left| \bdlambda_{\le s-\delta}\right|-\sum_{k=c}^{s-\delta-1} \lambda_k &\ge m(\mu_a^0+b(s-\delta-c_0))-(s-\delta-c)(mb-1)\\
        &> m(\mu_a^0+b(c-c_0))
    \end{align*}
    when $c_0\le c<s-\delta$.  
    Consequently, it suffices to require \eqref{eqn:15} for $c=s-\delta$ solely.
    
    Now, the multigraded module $I_a(\calF_{\bdz}^b)^{(m)}/I_a(\calF_{\bdz}^b)^m$ is minimally generated by the images of $\bdz^{\bdlambda}$ where the partition $\bdlambda$ satisfies
    \[
        |\bdlambda_{\le s-\delta}|= m(\mu_a^0+b(s-\delta-c_0))\quad\text{and}\quad \lambda_{s-\delta}=\lambda_{s-\delta+1}=\cdots =\lambda_s<mb.
    \]
    We collect these partitions into the set $\Lambda(I,m)$. 

    The final piece of the proof is the well-known fact that $|\frakS_n\cdot
    \bdz^{\bdlambda}|=\frac{s!}{\type(\bdlambda)!}$.
\end{proof}

\begin{Example}
    Consider the ideal $I=I_7(z_1^2\cdots z_5^2)$ in $T=\KK[z_1,\dots,z_5]$.
    Then $\delta=1$, $c_0=2$ and $\mu_a^0=1$. Whence, 
    \begin{align*}
        \Lambda(I,2)=\{\bdlambda: |\bdlambda_{\le 4}|=10 \text{ and } \lambda_4<4\}=\Set{(2,2,3,3,3),(1,3,3,3,3)}.
    \end{align*}
    And consequently, $\sdefect(I,2)=\frac{5!}{2!3!}+\frac{5!}{1!4!}=15$,
    agreeing with the computation by \texttt{Macaulay2} \cite{M2} via the
    \texttt{SymbolicPowers} package.
\end{Example}

\subsection{General case}

It is time to shift back our focus to symbolic powers of generalized star
configurations of generic forms. 

\begin{Proposition}
    \label{thm:mingens-strongly-generic} 
    With the assumptions in \Cref{prop:symbolic-power-strongly-generic}, let $\varphi:T=\KK[z_1,\dots,z_s]\to R$ be the homomorphism induced by $z_i\mapsto f_i$ for each $i$.  For each fixed positive integer $m$, suppose that $\bdF_{\bullet}$ is a graded minimal free resolution of $T/I_a(\calF_{\bdz}^b)^{(m)}$ for $\calF_{\bdz}=\{z_1,\dots,z_s\}\subseteq T$. Then, $I_a(\calF^b)^{(m)}=\varphi(I_a(\calF_{\bdz}^b)^{(m)})$ and $\bdF_{\bullet}\otimes_{T} R$ is a graded minimal free resolution of $R/I_a(\calF^b)^{(m)}$. In particular, a minimal generating set of $I_a(\calF^b)^{(m)}$ is given by
    \begin{equation}   
        \Set{\varphi(g)| g\in G(I_a(\calF_{\bdz}^{b})^{(m)})}.
        \label{eqn:generating-set} 
    \end{equation} 
    Here, $G(I_a(\calF_{\bdz}^b)^{(m)})$ is the minimal monomial generating set of the corresponding monomial ideal. 
\end{Proposition}

\begin{proof}
    We prove by an induction on $s-\hat{c}$ where $\hat{c}\coloneqq \max\{c:\text{$\calF$ is $c$-generic}\}$. Obviously, $s-\delta\le \hat{c}\le s-1$. If $s-\hat{c}=1$, $\calF$ forms a regular sequence. Whence, $\varphi$ is flat and the claim is clear by \Cref{key-facts} and \Cref{prop:decomposition-strongly-generic-special}.

    When $s-\hat{c}>1$, let $x_0$ be a new variable over $R$. For each $i$, let $f_i'$ be a general $d$-form in the ideal $(f_i,x_0)\subset R[x_0]$. Now, consider the new homomorphism $\gamma: T\to R[x_0]$ induced by $z_i\mapsto f_i'$. Notice that $\calF'\coloneqq\{f_1',\dots,f_s'\}$ is a set of $(\hat{c}+1)$-generic forms in $R[x_0]$. Whence, by induction, $\bdF_{\bullet}\otimes_T R[x_0]$ is a graded minimal free resolution of $R[x_0]/I_a( (\calF')^b)^{(m)}$.  Meanwhile, we have the graded isomorphism
    \[
        R[x_0]/(I_a( (\calF')^b)^{(m)},x_0) \cong R/I_a(\calF^b)^{(m)}.
    \]
    Thus, the last piece of the proof is to show that $x_0$ is a non-zero-divisor of $R[x_0]/I_a( (\calF')^b)^{(m)}$. 

    For this, we take arbitrary $g\in R[x_0]$ and assume that $x_0 g\in I_a( (\calF')^b)^{(m)}$.  By \Cref{prop:decomposition-strongly-generic-special} and \cite[Theorem 3.6]{MR3683102}, this is equivalent to saying that $x_0g\in I_{c,\calF'}^{(m(\mu_a^0+b(c-c_0)))}$ for each $c$ with $c_0\le c\le s-\delta$. However, it is shown in the proof of \cite[Theorem 3.3]{MR3683102} that $x_0$ is a non-zero-divisor of $R[x_0]/I_{c,\calF'}^{(m(\mu_a^0+b(c-c_0)))}$.  Therefore, $g\in I_{c,\calF'}^{(m(\mu_a^0+b(c-c_0)))}$ for each $c$.  Consequently, $g\in I_a( (\calF')^b)^{(m)}$.  This means that $x_0$ is indeed a non-zero-divisor, as expected.
\end{proof}

\begin{Theorem}
    \label{Cor:sdefect-rho-strongly-generic}
    With the assumptions in \Cref{prop:symbolic-power-strongly-generic}, we take further $\calF_{\bdz}=\{z_1,\dots,z_s\}$ in $T=\KK[z_1,\dots,z_s]$.  Then, we have the following properties.
    \begin{enumerate}[a]
        \item $\projdim_R(R/I_a(\calF^b)^{(m)})=s-\delta$ and $\reg_R(R/I_a(\calF^b)^{(m)})=md(\mu_a^0+b(s-c_0))-1$ for each positive integer $m$.
        \item The resurgence $\rho(I_{a}(\calF^b))$ coincides with the $\rho(I_a(\calF_{\bdz}^b))$ given in \Cref{thm:rho-monomial}.
        \item For each positive integer $m$, the symbolic defect $\sdefect(I_a(\calF^b),m)$ coincides with the $\sdefect(I_a(\calF_\bdz^{b}),m)$ given in \Cref{thm:symbolic-defect-monomial}.
    \end{enumerate}
\end{Theorem}

\begin{proof}
    \begin{enumerate}[a] 
        \item This follows from \Cref{thm:mingens-strongly-generic} and \Cref{cor:projdim-symbolic-power}. 
        \item It is clear that $\alpha(I_a(\calF^b))=d\alpha(I_a(\calF_{\bdz}^b))$.  And by the description of the minimal generating set in \eqref{eqn:generating-set}, it is clear that $\alpha(I_a(\calF^b)^{(m)})=d\alpha(I_a(\calF_{\bdz}^b)^{(m)})$.  Therefore, 
            \[
                \rho(I_a(\calF^b))\ge \frac{\alpha(I_a(\calF^b))}{\widehat{\alpha}(I_{a}(\calF^b))}=
                \frac{d\alpha(I_a(\calF_{\bdz}^b))}{d\widehat{\alpha}(I_a(\calF_{\bdz}^b))}
                =\rho(I_a(\calF_{\bdz}^b)).
            \] 
            For the reverse direction, we notice that $I_a(\calF_{\bdz}^b)^{(m)}\subseteq I_a(\calF_{\bdz}^b)^r$ will imply $I_a(\calF^b)^{(m)}\subseteq I_a(\calF^b)^r$, by the description in \eqref{eqn:generating-set}.  Consequently, $\rho(I_a(\calF^b))\le\rho(I_a(\calF_{\bdz}^b))$.  And this establishes the equality. 

        \item Let $\varphi$ be the homomorphism from $T$ to $R$, induced by $z_i\mapsto f_i$ for each $i$. Now, take arbitrary $\bdz^{\bdlambda}\in G(I_a(\calF_\bdz^b)^{(m)})$. 
            Without loss of generality, we may assume that $\bdlambda$ is an increasing sequence and hence $\bdlambda\in \Lambda(I_a(\calF_\bdz^b)^{(m)})$.
            If $\bdz^{\bdlambda} \in I_a(\calF_\bdz^b)^{m}$, then it is clear that $\varphi(\bdz^{\bdlambda})\in I_a(\calF^b)^{m}$. On the other hand, if $\bdz^{\bdlambda} \notin I_a(\calF_\bdz^b)^{m}$, then by the proof of \Cref{thm:symbolic-defect-monomial}, we have 
            \[
                |\bdlambda|\le m(\mu_a^0+b(s-\delta-c_0))+\delta(mb-1)=ma-\delta<ma.
            \]
            Thus 
            $\deg(\varphi(\bdz^{\bdlambda}))=d|\bdlambda|<dma=\alpha(I_a(\calF^b)^m)$, implying $\varphi(\bdz^{\bdlambda})\notin I_a(\calF^b)^{m}$. 
            
            Now, $I_a(\calF^b)^{(m)}/I_a(\calF^b)^m$ is generated by the set
            \[
                \Set{\varphi(\bdz^{\bdlambda}):\bdz^{\bdlambda}\in G(I_a(\calF_\bdz^b)^{(m)}) \text{ and } \bdz^{\bdlambda} \notin I_a(\calF_\bdz^b)^{m}}
            \]
            by \Cref{thm:mingens-strongly-generic}.
            The minimality of this generating set comes from \Cref{thm:mingens-strongly-generic} together with the degree reasons stated above.
            \qedhere
    \end{enumerate}
\end{proof}

\begin{Question}
    By \cite[Corollary 5.6]{arXiv:1907.04288}, we now have a closed formula for
    the Betti tables of the uniform $a$-fold product ideal $I_a(z_1^b\cdots
    z_s^b)$, although it is understandably complicated. What can be said
    regarding the general case for $I=I_a(z_1^{m_1}\cdots z_s^{m_s})$?
\end{Question}

\section{$*$-symbolic powers of general uniform $a$-fold product ideal} 

We have seen in \Cref{prop:projdim-monomial} that the graded maximal ideal $\frakm$ of $R=\KK[x_0,\dots,x_n]$ is an associate prime of $I=I_a(\calF^b)$ when $a\le (b-1)s+1$.  Whence, $I^{(m)}=I^m$ for all positive integer $m$. Of course, this is not very interesting for considering containment problem of symbolic powers. 

On the other hand, there is another approach in the literature for treating
symbolic objects, namely, one only takes intersections over minimal primes of
$I$ in \eqref{eqn:symbolic-power}; see for instance \cite[Definition
4.3.22]{MR3362802}. To avoid confusion in notation, we will call it the
\emph{$m$-th $*$-symbolic power} of $I$, namely, we will have
\[
    I^{(m)_*}\coloneqq \bigcap_{\frakp\in\Min(I)} (I^m R_\frakp \cap R).
\]
Obviously, this notion coincides with the standard one when $I$ has no embedded
associate prime. And this is the case for the defining ideal of star
configurations of hypersurfaces. 

Now, it is time to deal with the containment problem of $*$-symbolic powers of
ideal $I_a(\calF^b)$ for $b\ge 2$ under \Cref{main-setting} and \Cref{setting-uniform-fold-product}.

Since the forms in $\calF$ are generic and $I_a(\calF^b)$ has positive dimension, it follows from \Cref{prop:sat} that $\Ht(I_a(\calF^b))=c_0$ and a prime ideal $\frakp$ is a minimal prime of $I_a(\calF^b)$ if and only if it is a minimal prime of some complete intersection ideal $\braket{f_{i_1},\dots,f_{i_{c_0}}}$ for a unique sequence $1\le i_1<\cdots<i_{c_0}\le s$. Furthermore, for this pair of prime ideal and $i$-sequence, one can verify directly that 
\[
    (I_a(\calF^b))^m R_{\frakp}= \braket{f_{i_1},\dots,f_{i_{c_0}}}^{m\mu_a^{0}}R_{\frakp}.
\]
Since $\braket{f_{i_1},\dots,f_{i_{c_0}}}$ is a complete intersection, its
power is Cohen--Macaulay and 
\[
    \Ass(R/\braket{f_{i_1},\dots,f_{i_{c_0}}})=
    \Ass(R/\braket{f_{i_1},\dots,f_{i_{c_0}}}^{m\mu_a^0}).
\]
This implies that
\begin{align}
    I_a(\calF^b)^{(m)_{*}}&= \bigcap_{1\le i_1<\cdots <i_{c_0}\le s} \left(\bigcap_{\frakp\in\Ass(R/\braket{f_{i_1},\dots,f_{i_{c_0}}})} \braket{f_{i_1},\dots,f_{i_{c_0}}}^{m\mu_a^0}R_{\frakp} \cap R\right) \notag\\
    &= \bigcap_{1\le i_1<\cdots <i_{c_0}\le s} \left(\bigcap_{\frakp\in\Ass(R/\braket{f_{i_1},\dots,f_{i_{c_0}}}^{m\mu_a^0})} \braket{f_{i_1},\dots,f_{i_{c_0}}}^{m\mu_a^0}R_{\frakp} \cap R\right) \notag \\
    &= \bigcap_{1\le i_1<\cdots <i_{c_0}\le s}  \braket{f_{i_1},\dots,f_{i_{c_0}}}^{m\mu_a^0}  \notag \\
    &=I_{c_0,\calF}^{(m\mu_a^0)}.
    \label{eqn:star-symbolic-power}
\end{align}

Recall that if $I$ is a nonzero graded ideal in $R$, then $\alpha(I)$ the least degree of nonzero forms in $I$.  We will in addition consider the \emph{$*$-Waldschmidt constant} 
\begin{equation}
    \widehat{\alpha}_{*}(I)\coloneqq \lim_{m\to \infty} \frac{\alpha(I^{(m)_{*}})}{m},
    \label{eqn:star-Waldschmidt-def}
\end{equation}
the \emph{$*$-resurgence}
\[
    \rho_{*}(I)=\sup\Set{\frac{m}{r}| I^{(m)_{*}}\nsubseteq I^r},
\]
and the \emph{$m$-th $*$-symbolic defect}
\[
    \sdefect_{*}(I,m)=\mu(I^{(m)_{*}}/ I^m).
\]

Based on the previous work of \cite{MR2629595}, \cite{MR3003727},
\cite{MR3390029} and \cite{arXiv:1907.08172}, we are able to talk about the
above concepts with respect to the uniform $a$-fold product ideal
$I_a(\calF^b)$ in the following.

\subsection{$*$-Waldschmidt constant and $*$-resurgence}
This subsection is devoted to the study of the $*$-resurgence of $I_a(\calF^b)$.

\begin{Remark}
    \label{rmk:star} 
    Notice that the containment $I\subseteq I^{(1)_{*}}$ may be strict. Thus, it is possible that $\alpha(I)\ne \alpha(I^{(1)_{*}})$. Nevertheless, we still have some familiar properties regarding the asymptotic quantities above.  We collect some pertinent preliminary facts about them here. The proofs of these facts are virtually the same as those in \cite[Lemma 8.2.2]{MR2555949} and \cite{MR2629595}, hence will be omitted here.  
    \begin{enumerate}[a] 
        \item 
            The limit in \eqref{eqn:star-Waldschmidt-def} exists and $\widehat{\alpha}_{*}(I)\ge 1$. 
%
        \item If $r\alpha(I)>\alpha(I^{(m)_{*}})$, then $I^r$ does not contain $I^{(m)_{*}}$.
        \item \label{item-d-star} 
            If $m/r<\alpha(I)/\widehat{\alpha}_{*}(I)$, then for all $t\gg 0$,
            $I^{rt}$ does not contain $I^{(mt)}$. In particular, $
            \alpha(I)/\widehat{\alpha}_{*}(I)\le \rho_{*}(I)$.
        \item If $I^{(m)_{*}}\subseteq I^r$, then $r\le m$.
    \end{enumerate}
\end{Remark}

\begin{Remark}
    We state additionally some useful facts regarding $I_{a}(\calF^b)$ for $\calF=\{f_1,\dots,f_s\}$ in $R=\KK[x_0,\dots,x_n]$. For that purpose, take $\calF_{\bdz}=\{z_1,\dots,z_s\}$ in $T=\KK[z_1,\dots,z_s]$.
    Notice that $I_{c_0,\calF}^{(m\mu_a^0)}$ is the specialization of $I_{c_0,\calF_{\bdz}}^{(m\mu_a^0)}$.

    If the forms in $\calF$ share a common degree $d$, then, via the description in \eqref{eqn:star-symbolic-power}, we will have 
    \[
        \alpha\left( I_a(\calF^b)^{(m)_{*}} \right)=
        d\alpha\left( I_a(\calF_{\bdz}^b)^{(m)_{*}} \right)=
        d\alpha\left( 
            I_{c_0,\calF_{\bdz}}^{(m\mu_a^0)}
        \right).
    \]
    Whence,
    \begin{align}
        \widehat{\alpha}_{*}(I_a(\calF^b))&=
        d\widehat{\alpha}_{*}(I_a(\calF_{\bdz}^b))=
        d\lim_{m\to \infty} \frac{\alpha\left( I_a(\calF_{\bdz}^b)^{(m)_{*}} \right)}{m} \notag\\
        &=d\mu_a^0\lim_{m\to \infty} \frac{\alpha\left( I_{c_0,\calF_{\bdz}}^{(m\mu_a^0)} \right)}{m\mu_a^0}=
        d\mu_a^0\widehat{\alpha}(I_{c_0,\calF_{\bdz}})=\frac{d\mu_a^0 s}{c_0}.
        \label{eqn:star-Waldschmidt-linear}
    \end{align}
    The last equality above uses the fact that $\widehat{\alpha}(I_{c_0,\calF_{\bdz}})=s/c_0$ by \cite[Theorem 7.5]{MR3566223}.

    Similarly, the containment $I_a(\calF_{\bdz}^b)^{(m)_{*}}\subseteq I_a(\calF_{\bdz}^b)^r$ implies the containment $I_a(\calF^b)^{(m)_{*}}\subseteq I_a(\calF^b)^r$. Consequently,
    \begin{equation}
        \rho_{*}(I_a(\calF^b))\le \rho_{*}(I_a(\calF_{\bdz}^b)).
        \label{eqn:rho-specialization}
    \end{equation} 
\end{Remark}

It is then natural to ask for the explicit value of $\rho_{*}(I_a(\calF_{\bdz}^b))$.

\begin{Proposition}
    \label{thm:star-rho-monomial}
    With the assumptions in \Cref{setting-uniform-fold-product}, we consider the set $\calF_{\bdz}=\{z_1,\dots,z_s\}$ in $T=\KK[z_1,\dots,z_s]$. Then, for the uniform $a$-fold product ideal $I=I_{a}(\calF_{\bdz}^b)$, we have
    \[
        \rho_{*}(I)=\frac{\alpha(I)}{\widehat{\alpha}_{*}(I)}=\frac{ac_0}{\mu_a^{0}s}.
    \]
\end{Proposition}

\begin{proof}
    It follows from item \ref{item-d-star} in \Cref{rmk:star} and equation \eqref{eqn:star-Waldschmidt-linear} that 
    \begin{equation}
        \rho_{*}(I)\ge \alpha(I)/\widehat{\alpha}_{*}(I)= \frac{ac_0}{\mu_a^0s }. 
        \label{eqn:rho-ge}
    \end{equation}
    It remains to prove that $\rho_{*}(I)$ is bounded above by the expected
    value.

    In the following, we first explore relations between $m$ and $r$ such that
    $I^{(m)_{*}}\subseteq I^r$, i.e., $I_{c_0,\calF_{\bdz}}^{(m\mu_a^0)}
    \subseteq I_{ar}(\calF_{\bdz}^{br})$. For this purpose, notice that these
    two ideals are symmetric.  And partitions
    $\bdlambda\in\calP(I_{c_0,\calF_{\bdz}}^{(m\mu_a^0)})$ are characterized
    by the requirement $|\bdlambda_{\le c_0}|\ge m\mu_a^0$.  Meanwhile,
    $\bdlambda\in\calP(I_a(\calF_{\bdz}^b)^r)$ if and only if $\sum_i
    \min(\lambda_i,br)\ge ar$.  

    Now, take arbitrary partition $\bdlambda\in
    \calP(I_{c_0,\calF_{\bdz}}^{(m\mu_{a}^0)})$. Furthermore, we may write
    $m\mu_{a}^{0}=q_0c_0+q_1$ with $1\le q_1\le c_0$. Then, $\lambda_{c_0}\ge
    q_0+1=\ceil{\frac{m\mu_a^0}{c_0}}$.  For the containment
    $I_a(\calF_{\bdz}^b)^{(m)}\subseteq I_a(\calF_{\bdz}^b)^r$ with $m\ge r\ge
    1$, we have two cases.
    \begin{enumerate}[a]
        \item Suppose that $q_0+1<br$.
            \begin{enumerate}[i]
                \item Suppose further that $q_0+1\le \lambda_{c_0}\le br$.  Whence, the minimum of $\sum_{i=1}^s \min(\lambda_i,br)$ in this case is simply $m\mu_a^0+ (s-c_0)(q_0+1)$, achieved at $\bdlambda=(q_0^{c_0-q_1},(q_0+1)^{s-c_0+q_1})$. Thus, the requirement for the containment is 
                    \[
                        m\mu_a^0+ (s-c_0)(q_0+1)\ge ar, 
                    \]
                    which is equivalent to asking for
                    \begin{equation*}
                        r\le \frac{s(q_0+1)+(q_1-c_0)}{a}.
                    \end{equation*}
                \item \label{condition-a-i}
                    Suppose instead that $q_0+1<br \le \lambda_{c_0}$, then the minimum of $\sum_{i=1}^s \min(\lambda_i,br)$ in this case is $(s-c_0+1)br$. Now, the requirement for the containment is
                    \[
                         (s-c_0+1)br \ge ar, 
                    \]
                    which holds automatically. Thus, there is no requirement in this subcase.
            \end{enumerate}
        \item Suppose that $br\le q_0+1$. Similar to the discussion in
            \ref{condition-a-i} above, the requirement is void in this case.
    \end{enumerate}

    To sum up, the above arguments implies that $I^{(m)_{*}}\nsubseteq I^{r}$  precisely when
    \[
        \ceil{\frac{m\mu_a^0}{c_0}}< br \quad \text{and}\quad r> \frac{s(q_0+1)+(q_1-c_0)}{a}.  
    \]
    Notice that for any fixed $m$, we always have $I^{(m)_{*}}\nsubseteq I^{r}$ for sufficiently large $r$. This means that the conditions above are not empty.  Whence,
    \[
        \frac{m}{r}<\frac{a((q_0+1)c_0+(q_1-c_0))}{\mu_a^0((q_0+1)s+(q_1-c_0))}
        \le \frac{ac_0}{\mu_a^0s},
    \]
    as $1\le q_1\le c_0$ and $c_0<s$. In particular, 
    \[
        \rho_{*}(I)\le \frac{ac_0}{\mu_a^0s}.
    \]
    Combining the inequality \eqref{eqn:rho-ge}, we obtain the desired formula for $\rho_{*}(I)$.
\end{proof}

Now, we are ready for the first main result of this section.

\begin{Theorem}
    With the assumptions in Settings \ref{main-setting} and \ref{setting-uniform-fold-product}, we assume further that the forms in $\calF=\{f_1,\dots,f_s\}$ share a common degree $d$ in $\KK[x_0,\dots,x_n]$. Then 
    \[
        \rho_{*}(I)=\frac{\alpha(I)}{\widehat{\alpha}_{*}(I)}=\frac{ac_0}{\mu_a^0s}
    \]
    for the ideal $I=I_a(\calF^b)$.
\end{Theorem}

\begin{proof}
    Take in addition $\calF_z=\{z_1,\dots,z_s\}$ in $T=\KK[z_1,\dots,z_s]$. Now, it suffices to notice that
    \[
        \frac{ac_0}{\mu_a^0s}= \frac{{\alpha}(I_a(\calF^b))}{\widehat{\alpha}_{*}(I_a(\calF^b))}\le \rho_{*}(I_a(\calF^b))\le \rho_{*}(I_a(\calF_{\bdz}^b))=\frac{ac_0}{\mu_a^0s},
    \]
    by combining the equations \eqref{eqn:star-Waldschmidt-linear},
    \eqref{eqn:rho-specialization} and \Cref{thm:star-rho-monomial}.
\end{proof}

\subsection{$*$-symbolic defect}
This subsection is devoted to the study of the $*$-symbolic defect of $I_a(\calF^b)$. We start with the following observation.

\begin{Lemma}
    [{\cite[Proposition 4.10(2)]{arXiv:1907.08172}}]
    \label{P-4.10}
    With the assumptions in Settings \ref{main-setting}, we further take $\calF_{\bdz}=\{z_1,\dots,z_s\}$ in $T=\KK[z_1,\dots,z_s]$. Let $\varphi: T\to R$ be the homomorphism induced by $z_i\mapsto f_i$ for each $i$. Then for each positive integer $m$, we have
    \[
        \Set{\varphi(g):g\in G\left(I_{c,\calF_{\bdz}}^{(m)}\right)}\cap I_{c,\calF}^m=\Set{\varphi(h)^m:h\in G(I_{c,\calF_{z}})}.
    \]
   In particular, we have $I_{c,\calF}^m \nsubseteq \braket{f_1,\dots,f_s}I_{c,\calF}^{(m)}$.
\end{Lemma}

Here is the key observation for the last strike.

\begin{Proposition}
    \label{prop:compare-power-star-symbolic-power}
    With the assumptions in Settings \ref{main-setting} and \ref{setting-uniform-fold-product}, we assume further that the forms in $\calF=\{f_1,\dots,f_s\}$ share a common degree $d$ in $\KK[x_0,\dots,x_n]$. Then the following statements are equivalent:
    \begin{enumerate}[a]
        \item \label{thm:star-sdefect-a-i}
            $I_{a}(\calF^b)^m \nsubseteq
            \braket{f_1,\dots,f_s}I_a(\calF^{b})^{(m)_{*}}$ for all positive integer $m$; 
        \item \label{thm:star-sdefect-a-ii}
            $I_{a}(\calF^b)^m \nsubseteq
            \braket{f_1,\dots,f_s}I_a(\calF^{b})^{(m)_{*}}$ for some positive integer $m$; 
        \item \label{thm:star-sdefect-a-iii}
            the parameter $a$ is a multiple of $b$.
    \end{enumerate}
\end{Proposition}

\begin{proof} 
    Firstly, we show that \ref{thm:star-sdefect-a-ii}$\Rightarrow$\ref{thm:star-sdefect-a-iii}. Let $\varphi: T\to R$ be the homomorphism induced by $z_i\mapsto f_i$ for each $i$.  Since $I_{a}(\calF^b)^m=\varphi(I_a(\calF_{\bdz}^b)^{m})$ while 
    \[
        \braket{f_1,\dots,f_s}I_a(\calF^b)^{(m)_{*}}
        =\varphi( \braket{z_1,\dots,z_s})\varphi(I_a(\calF_{\bdz}^b)^{(m)_{*}})
        =\varphi( \braket{z_1,\dots,z_s}I_a(\calF_{\bdz}^b)^{(m)_{*}})
        ,
    \]
    we will have $I_{a}(\calF_{\bdz}^b)^m \nsubseteq \braket{z_1,\dots,z_s}I_a(\calF_{\bdz}^{b})^{(m)_{*}}$. But it is clear that $I_{a}(\calF_{\bdz}^b)^m \subseteq I_a(\calF_{\bdz}^{b})^{(m)_{*}}$. Now, $I_{a}(\calF_{\bdz}^b)^m \nsubseteq \braket{z_1,\dots,z_s}I_a(\calF_{\bdz}^{b})^{(m)_{*}}$ if and only if the intersection of minimal monomial generating sets $G(I_{a}(\calF_{\bdz}^b)^m)\cap G(I_a(\calF_{\bdz}^{b})^{(m)_{*}})\ne \varnothing$. Take any monomial $\bdz^{\bdlambda}$ in that intersection. Without loss of generality, we may assume that $\lambda_1\le \cdots \le \lambda_s$. Now $\bdlambda\in \Lambda(I_a(\calF_{\bdz}^b)^m)$ and $|\bdlambda|=am$. On the other hand, since $I_a(\calF_{\bdz}^{b})^{(m)_{*}}=I_{c_0,\calF_{\bdz}}^{(m\mu_a^0)}$, $\bdlambda\in \Lambda(I_{c_0,\calF_{\bdz}}^{(m\mu_a^0)})$. Consequently, $|\bdlambda_{\le c_0}|=m\mu_a^0$ and $\lambda_{c_0}=\cdots=\lambda_s$. Now, $|\bdlambda|=am=m\mu_a^0+(s-c_0)\lambda_{c_0}$. As $c_0<s$ and $a-\mu_a^0=b(s-c_0)$, this forces $mb=\lambda_{c_0}$.

    Notice that $\lambda_{c_0}\le |\bdlambda_{\le c_0}|=m\mu_a^0$ while we always have $\mu_a^{0}\le b$. Therefore, the non-emptiness of the intersection forces $b=\mu_a^0$, i.e., $a$ is a multiple of $b$. And this shows the implication \ref{thm:star-sdefect-a-ii}$\Rightarrow$\ref{thm:star-sdefect-a-iii}. 

    When $a=kb$ is a multiple of $b$, then $\mu_a^0=b$, and $c_0=s-k+1$. Now,
    \[
        I_a(\calF^b)=I_{bk}(\calF^b)=I_k(\calF)^b=I_{s-c_0+1}(\calF)^b=I_{c_0,\calF}^{b}
    \]
    by \eqref{eqn:I-I} and \eqref{eqn-power}.
    Consequently, $I_a(\calF^b)^m=I_{c_0,\calF}^{mb}$.  Meanwhile,  $I_a(\calF^b)^{(m)_{*}}=I_{c_0,\calF}^{(mb)}$ by \eqref{eqn:star-symbolic-power}.  Thus, \Cref{P-4.10} implies that $I_{a}(\calF^b)^m \nsubseteq \braket{f_1,\dots,f_s}I_a(\calF^{b})^{(m)_{*}}$. And this shows the implication \ref{thm:star-sdefect-a-iii}$\Rightarrow$\ref{thm:star-sdefect-a-i}.

    The implication \ref{thm:star-sdefect-a-i}$\Rightarrow$\ref{thm:star-sdefect-a-ii} is trivial.
\end{proof}

Here comes the final result of this paper.

\begin{Theorem}
    With the assumptions in Settings \ref{main-setting} and \ref{setting-uniform-fold-product}, we assume further that $m$ is a positive integer and the forms in $\calF=\{f_1,\dots,f_s\}$ share a common degree $d$ in $R=\KK[x_0,\dots,x_n]$.     
    \begin{enumerate}[a]
        \item 
            If the parameter $a$ is not a multiple of $b$, then $\sdefect_{*}(I_a(\calF^b),m)$ is given by
            \[
                \sum_{B=\{b_1<\cdots<b_h\}\subseteq[c]}\left| \SS_B \right| \binom{s}{c_0-b_h}\binom{s-c_0+b_h}{b_h-b_{h-1}}\binom{s-c_0+b_{h-1}}{b_{h-1}-b_{h-2}}\cdots \binom{s-c_0+b_2}{b_2-b_1},
            \]
            where $\SS_B$ is the set of all distinct positive solutions to the Diophantine equation $b_1x_1+\cdots+b_hx_h=m\mu_a^0$.
        \item 
            If $a=kb$ is a multiple of $b$, then $I_a(\calF^b)=I_{s-c_0+1}(\calF)^b$. Whence, $\sdefect_{*}(I_a(\calF^b),m)$ is given by
            \[
                \sum_{B=\{b_1<\cdots<b_h\}\subseteq[c]}\left| \SS_B \right| \binom{s}{c_0-b_h}\binom{s-c_0+b_h}{b_h-b_{h-1}}\binom{s-c_0+b_{h-1}}{b_{h-1}-b_{h-2}}\cdots \binom{s-c_0+b_2}{b_2-b_1}-\binom{s}{c_0-1},
            \]
            where $\SS_B$ is defined as above with $\mu_a^0=b$.
    \end{enumerate}
\end{Theorem} 

\begin{proof} 
    If $a$ is not a multiple of $b$, then 
    $I_{a}(\calF^b)^m \subseteq
            \braket{f_1,\dots,f_s}I_a(\calF^{b})^{(m)_{*}}$ by \Cref{prop:compare-power-star-symbolic-power}. Therefore, 
    $\sdefect_{*}(I_a(\calF^b),m)$ is simply the minimal number of generators of $I_a(\calF^b)^{(m)_{*}}=I_{c_0,\calF}^{(m\mu_a^0)}$ which is given in \cite[Corollary 4.12(1)]{arXiv:1907.08172}.

    Similarly, if $a=kb$ is a multiple of $b$, then we have already seen that $I_a(\calF^b)=I_{c_0,\calF}^b$.  Since 
    \[
        \Min(I_a(\calF^b))=\Min(I_{c_0,\calF}^b)=\Min(I_{c_0,\calF})=\Ass(R/I_{c_0,\calF}),
    \]
    this implies that $\sdefect_{*}(I_a(\calF^b),m)=\sdefect(I_{c_0,\calF},mb)$ and the latter is given in \cite[Corollary 4.12(2)]{arXiv:1907.08172}.
\end{proof}

\begin{Remark}
    When $s-c_0+1$ or $m$ is small, a more explicit description of the $*$-symbolic defect is available in \cite[Section 4.2]{arXiv:1907.08172} via the information above. 
\end{Remark}

\begin{acknowledgment*}
    The authors want to express sincere thanks to Toh\v{a}neanu and Xie for helpful discussions. The second author is partially supported by the ``Anhui Initiative in Quantum Information Technologies'' (No.~AHY150200) and the ``Fundamental Research Funds for the Central Universities''.
\end{acknowledgment*}


\end{document}